\documentclass{article}
\usepackage[utf8]{inputenc}
\usepackage{comment}
\usepackage{amssymb,amsmath,amsthm,graphicx,xcolor,mathtools,enumerate,mathrsfs}
\usepackage{enumitem,thmtools}
\usepackage[left=4cm, right=4cm, top=4cm, bottom=2.25cm]{geometry}
\usepackage[d]{esvect}
\usepackage{indentfirst}
\usepackage{hyperref}
\usepackage{subcaption}

\theoremstyle{plain}
\newtheorem{theorem}{Theorem}[section]

\newtheorem{proposition}[theorem]{Proposition}

\newtheorem{lemma}[theorem]{Lemma}

\newtheorem{claim}[theorem]{Claim}

\theoremstyle{definition}
\newtheorem{definition}[theorem]{Definition}

\newtheorem{question}[theorem]{Question}



\newcommand{\red}[1]{#1_{\textup{red}}}
\newcommand{\blue}[1]{#1_{\textup{blue}}}

\def\eps{\varepsilon}
\newcommand\ceil[1]{\left\lceil#1\right\rceil}
\newcommand\floor[1]{\left\lfloor#1\right\rfloor}
\newcommand\cT{S}
\renewcommand{\tau}{s}
\allowdisplaybreaks

\definecolor{DarkDesaturatedBlue}{HTML}{3A3556}
\definecolor{VividOrange}{HTML}{F15918}
\definecolor{PureOrange}{HTML}{FFBA00}
\definecolor{LightGrayishPink}{HTML}{EEC5D5}
\definecolor{VerySoftBlue}{HTML}{B5AFDB}

\usepackage{tikz}
\usetikzlibrary{math,patterns,patterns.meta}
\usetikzlibrary{snakes,calc,decorations.pathmorphing,through}
\tikzset{snake it/.style={decorate, decoration=snake}}

\makeatletter
\newcommand{\labelinthm}[1]{%
    \label{temp#1}
    \protected@write \@auxout {}{\string \newlabel{#1}{{\emph{\ref{temp#1}}}{\thepage}{\emph{\ref{temp#1}}}{temp#1}{}} }%
  }
  \makeatother
\newcounter{propcounter}

\title{Asymmetric Ramsey numbers of trees}

\begin{document}
\author{Jun Yan\thanks{Mathematical Institute, University of Oxford, Oxford OX2 6GG, UK. Email: \url{jun.yan@maths.ox.ac.uk}. Supported by
ERC Advanced Grant 883810.}}

\maketitle

\begin{abstract}
Let $n\geq\nu$, let $T$ be an $n$-vertex tree with bipartition class sizes $t_1\geq t_2$, and let $S$ be a $\nu$-vertex tree with bipartition class sizes $\tau_1\geq\tau_2$. Using four natural constructions, we show that the Ramsey number $R(T,S)$ is lower bounded by $\underline{R}(T,S)=\max\{n+\tau_2,\nu+\min\{t_2,\nu\},\min\{2t_1,2\nu\},2\tau_1\}-1$. 

Our main result shows that there exists a constant $c>0$, such that for all sufficiently large integers $n\geq\nu$, if (i) $\Delta(T)\leq cn/\log n$ and $\Delta(S)\leq c\nu/\log\nu$, (ii) $\tau_2\geq t_2$, and (iii) $\nu\geq t_1$, then $R(T,S)=\underline{R}(T,S)$. In particular, this determines the exact Ramsey numbers for a large family of pairs of trees. We also provide examples showing that $R(T,S)$ can exceed $\underline{R}(T,S)$ if any one of the three assumptions (i), (ii), and (iii) is removed.  
\end{abstract}

\section{Introduction}\label{sec:intro}
The Ramsey number of two graphs $G_1$ and $G_2$, denoted by $R(G_1,G_2)$, is the smallest positive integer $N$ such that every red/blue edge colouring of the $N$-vertex complete graph $K_N$ contains either a red copy of $G_1$ or a blue copy of $G_2$. For brevity, $R(G)$ is used to denote $R(G,G)$. The existence of $R(G_1,G_2)$ follows from the foundational result of Ramsey~\cite{Ra} 
in 1930, but determining the exact value of $R(G_1,G_2)$, or even just providing good bounds, has since proved extremely challenging. 

To muster any hope of determining $R(G_1,G_2)$ exactly, usually at least one of $G_1$ and $G_2$ needs to be sparse, with some natural candidates being trees and cycles.
Denote the $n$-vertex path by $P_{n-1}$. Gerencs\'er and Gy\'arf\'as~\cite{GG} 
proved in 1967 that for all $n\geq m\geq2$, $R(P_{n-1},P_{m-1})=n+\floor{m/2}-1$. For stars, it was determined by Burr and Roberts~\cite{BR} in 1973 that $R(K_{1,n-1},K_{1,m-1})=m+n-3$ if $m$ and $n$ are both even, and $R(K_{1,n-1},K_{1,m-1})=m+n-2$ otherwise. For cycles, Faudree and Schelp~\cite{FS} and Rosta~\cite{R} independently proved in the early 1970s that, aside from a few small exceptions, for all $n\geq m$, $R(C_n,C_m)=2n-1$ if $m$ is odd, $R(C_n,C_m)=n+m/2-1$ if both $n$ and $m$ are even, and $R(C_n,C_m)=\max\{n+m/2-1,2m-1\}$ if $n$ is odd and $m$ is even.

In this paper, we consider the Ramsey number $R(T,\cT)$ of two general trees $T$ and $\cT$. Even in the symmetric case when $T=\cT$, the exact value of $R(T)$ is only known for a few very specific families of trees until very recently. There is, however, a general lower bound given by Burr~\cite{B} in 1974. Let the bipartition class sizes of $T$ be $t_1\geq t_2\geq2$, and consider the following red/blue colourings of the complete graph $G$ on $N$ vertices.
\stepcounter{propcounter}
\begin{enumerate}[label = \textbf{\Alph{propcounter}\arabic{enumi}}]
    \item\label{oldlow:1} $N=t_1+2t_2-2$, $G$ consists of two disjoint red cliques $A$ and $B$ of sizes $t_1+t_2-1$ and $t_2-1$, respectively, with all edges between them blue.  
    \item\label{oldlow:2} $N=2t_1-2$, $G$ consists of two disjoint red cliques $A$ and $B$ both of size $t_1-1$, with all edges between them blue. 
\end{enumerate}

The constructions in both~\ref{oldlow:1} and~\ref{oldlow:2} contain no red copy of $T$, as the red connected components all have sizes smaller than $|T|$. The graph in~\ref{oldlow:1} contains no blue copy of $T$ as $B$ is not large enough to accommodate the smaller bipartition class of $T$, while the graph in~\ref{oldlow:2} contains no blue copy of $T$ as neither $A$ nor $B$ is large enough to accommodate the larger bipartition class of $T$. Together, these two examples give the general lower bound
\begin{equation}\label{eq:oldlower}
R(T)\geq \max\{t_1+2t_2,2t_1\}-1.\tag{$\dagger$}
\end{equation}

Burr conjectured~\cite{B} that this bound is tight for every tree $T$ with $t_1\geq t_2\geq2$. However, this was disproved in 1979 by Grossman, Harary, and Klawe~\cite{GHK} for certain trees called \emph{double stars}. For any $t_1\geq t_2\geq2$, let $D_{t_1,t_2}$ be the tree formed by joining the central vertices of the stars $K_{1,t_1-1}$ and $K_{1,t_2-1}$ with an edge, noting that $D_{t_1,t_2}$ has bipartition class sizes $t_1$ and $t_2$.
Grossman, Harary, and Klawe~\cite{GHK} showed that if $t_1\geq3t_2-2\geq10$, then $R(D_{t_1,t_2})=2t_1$, and thus the bound (\ref{eq:oldlower}) can be off by 1. Burr's conjecture was much more strongly disproved in 2016 by Norin, Sun, and Zhao~\cite{NSZ}, who showed in particular that $R(D_{2t,t})\ge (4.2-o(1))t$, and therefore the bound (\ref{eq:oldlower}) can be off by a multiplicative factor.

However, all known counterexamples to Burr's conjecture have large maximum degrees, so perhaps it could be salvaged by imposing a suitable maximum degree condition on $T$. Towards this, Haxell, \L uczak, and Tingley~\cite{HLT} showed in 2002 that the bound (\ref{eq:oldlower}) is approximately tight for trees with up to small linear maximum degrees. More precisely, they showed that for every $\mu>0$, there exists some $c>0$, such that any $n$-vertex tree $T$ with maximum degree $\Delta(T)\leq cn$ and bipartition class sizes $t_1\geq t_2$ satisfies $R(T)\le(1+\mu)\max\{t_1+2t_2,2t_1\}$. 

In a recent breakthrough, Montgomery, Pavez-Sign{\'e}, and Yan~\cite{MPY} showed that Burr's conjecture is true, given a suitable linear maximum degree condition. In particular, this result determines the exact Ramsey numbers of a large family of trees.  
\begin{theorem}[{\cite[Theorem 1.1]{MPY}}]\label{thm:oldmain}
There exists a constant $c>0$ such that the following holds.
Any $n$-vertex tree $T$ with $\Delta(T)\le cn$ and bipartition classes of sizes $t_1\geq t_2$ satisfies $R(T)=\max\{2t_1,t_1+2t_2\}-1.$
\end{theorem}

In this paper, building on ideas in~\cite{HLT} and~\cite{MPY}, we prove an asymmetric version of Theorem~\ref{thm:oldmain}, which determines the exact Ramsey numbers $R(T,\cT)$ for a large family of pairs of trees $(T,\cT)$. 

Let $T$ be an $n$-vertex tree with bipartition class sizes $t_1\geq t_2$. Let $\cT$ be a $\nu$-vertex tree with bipartition class sizes $\tau_1\geq\tau_2$. Without loss of generality, assume that $n\geq\nu$. Motivated by Burr's constructions~\ref{oldlow:1} and~\ref{oldlow:2} in~\cite{B} that lead to the general lower bound~(\ref{eq:oldlower}), we now provide a general lower bound for $R(T,\cT)$ by considering the following red/blue coloured complete graphs $G$ on $N$ vertices that contain neither a red copy of $T$ nor a blue copy of $\cT$ (see Figure~\ref{figure:low}). 
\stepcounter{propcounter}
\begin{enumerate}[label = \textbf{\Alph{propcounter}\arabic{enumi}}]
    \item\label{low:1} $N=n+\tau_2-2$, $G$ consists of two disjoint red cliques $A$ and $B$ of sizes $n-1$ and $\tau_2-1$, respectively, with all edges between them blue.  
    \item\label{low:2} $N=\nu+\min\{t_2,\nu\}-2$, $G$ consists of two disjoint blue cliques $A$ and $B$ of sizes $\nu-1$ and $\min\{t_2,\nu\}-1$, respectively, with all edges between them red. 
    \item\label{low:3} $N=2\min\{t_1,\nu\}-2$, $G$ consists of two disjoint blue cliques $A$ and $B$ both of size $\min\{t_1,\nu\}-1$, with all edges between them red.
    \item\label{low:4} $N=2\tau_1-2$, $G$ consists of two disjoint red cliques $A$ and $B$ both of size $\tau_1-1$, with all edges between them blue.  
\end{enumerate}
\begin{figure}[h]
\centering
\begin{subfigure}{0.49\textwidth}
\centering
\begin{tikzpicture}[scale=.8]
\def\spacer{3};
\def\Ahgt{1.25};
\def\Bhgt{0.75};
\coordinate (U_1) at (0,0);
\coordinate (U_2) at (\spacer,0);

\draw[blue,fill=blue!50] ($(U_1)+(0,\Ahgt)$) -- ($(U_1)-(0,\Ahgt)$) -- ($(U_2)-(0,\Bhgt)$) -- ($(U_2)+(0,\Bhgt)$) -- cycle;

\draw[white,fill=white] (U_1) circle [y radius=\Ahgt cm,x radius=1cm];
\draw[white,fill=white] (U_2) circle [y radius=\Bhgt cm,x radius=0.75cm];
\draw[red,fill=red!50] (U_1) circle [y radius=\Ahgt cm,x radius=1cm];
\draw[red,fill=red!50] (U_2) circle [y radius=\Bhgt cm,x radius=0.75cm];
\draw (U_1) node {$U_1$};
\draw (U_2) node {$U_2$};

\draw ($(U_1)+(0,-1.8)$) node {$t_1+t_2-1$};
\draw ($(U_2)+(0,-1.8)$) node {$\tau_2-1$};
\end{tikzpicture}
\caption{\ref{low:1}}
\end{subfigure}
\begin{subfigure}{0.49\textwidth}
\centering
\begin{tikzpicture}[scale=.8]
\def\spacer{3};
\def\Ahgt{1.25};
\def\Bhgt{0.75};
\coordinate (U_1) at (0,0);
\coordinate (U_2) at (\spacer,0);

\draw[red,fill=red!50] ($(U_1)+(0,\Ahgt)$) -- ($(U_1)-(0,\Ahgt)$) -- ($(U_2)-(0,\Bhgt)$) -- ($(U_2)+(0,\Bhgt)$) -- cycle;

\draw[white,fill=white] (U_1) circle [y radius=\Ahgt cm,x radius=1cm];
\draw[white,fill=white] (U_2) circle [y radius=\Bhgt cm,x radius=0.75cm];
\draw[blue,fill=blue!50] (U_1) circle [y radius=\Ahgt cm,x radius=1cm];
\draw[blue,fill=blue!50] (U_2) circle [y radius=\Bhgt cm,x radius=0.75cm];

\draw (U_1) node {$U_1$};
\draw (U_2) node {$U_2$};

\draw ($(U_1)+(0,-1.8)$) node {$\nu-1$};
\draw ($(U_2)+(0,-1.8)$) node {$\min\{t_2,\nu\}-1$};
\end{tikzpicture}
\caption{\ref{low:2}}
\end{subfigure}

\vspace{0.3cm}
\begin{subfigure}{0.49\textwidth}
\centering
\begin{tikzpicture}[scale=.8]
\def\spacer{3};
\def\Ahgt{1};
\def\Bhgt{1};
\coordinate (U_1) at (0,0);
\coordinate (U_2) at (\spacer,0);

\draw[red,fill=red!50] ($(U_1)+(0,\Ahgt)$) -- ($(U_1)-(0,\Ahgt)$) -- ($(U_2)-(0,\Bhgt)$) -- ($(U_2)+(0,\Bhgt)$) -- cycle;

\draw[white,fill=none] (U_1) circle [y radius=\Ahgt cm,x radius=0.75cm];
\draw[white,fill=none] (U_2) circle [y radius=\Bhgt cm,x radius=0.75cm];
\draw[blue,fill=blue!50] (U_1) circle [y radius=\Ahgt cm,x radius=0.75cm];
\draw[blue,fill=blue!50] (U_2) circle [y radius=\Bhgt cm,x radius=0.75cm];
\draw (U_1) node {$U_1$};
\draw (U_2) node {$U_2$};

\draw ($(U_1)+(0,-1.8)$) node {$\min\{t_1,\nu\}-1$};
\draw ($(U_2)+(0,-1.8)$) node {$\min\{t_1,\nu\}-1$};
\end{tikzpicture}
\caption{\ref{low:3}}
\end{subfigure}
\begin{subfigure}{0.49\textwidth}
\centering
\begin{tikzpicture}[scale=.8]
\def\spacer{3};
\def\Ahgt{1};
\def\Bhgt{1};
\coordinate (U_1) at (0,0);
\coordinate (U_2) at (\spacer,0);

\draw[blue,fill=blue!50] ($(U_1)+(0,\Ahgt)$) -- ($(U_1)-(0,\Ahgt)$) -- ($(U_2)-(0,\Bhgt)$) -- ($(U_2)+(0,\Bhgt)$) -- cycle;

\draw[white,fill=none] (U_1) circle [y radius=\Ahgt cm,x radius=0.75cm];
\draw[white,fill=none] (U_2) circle [y radius=\Bhgt cm,x radius=0.75cm];
\draw[red,fill=red!50] (U_1) circle [y radius=\Ahgt cm,x radius=0.75cm];
\draw[red,fill=red!50] (U_2) circle [y radius=\Bhgt cm,x radius=0.75cm];
\draw (U_1) node {$U_1$};
\draw (U_2) node {$U_2$};

\draw ($(U_1)+(0,-1.8)$) node {$\tau_1-1$};
\draw ($(U_2)+(0,-1.8)$) node {$\tau_1-1$};
\end{tikzpicture}
\caption{\ref{low:4}}
\end{subfigure}
\caption{The lower bound constructions~\ref{low:1}--\ref{low:4}.}
\label{figure:low}
\end{figure}
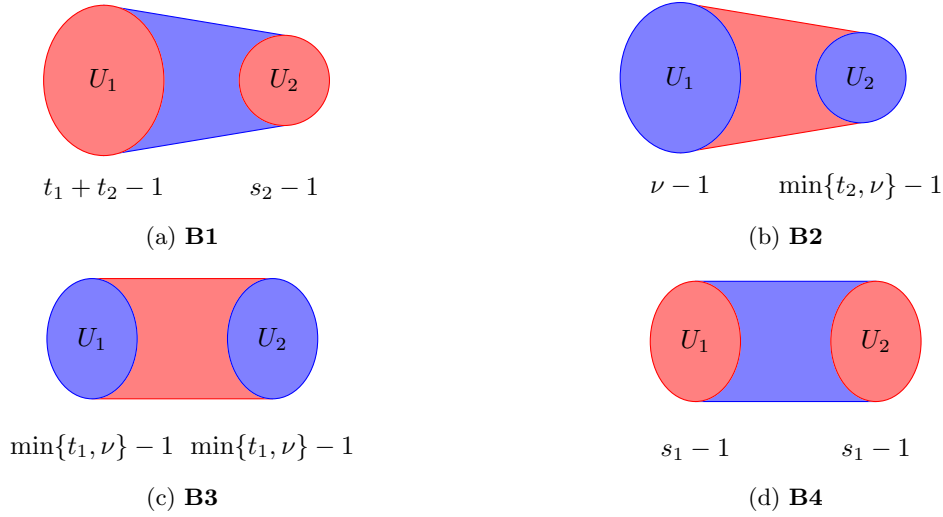

Similar to~\ref{oldlow:1} and~\ref{oldlow:2}, it is easy to see that in each of~\ref{low:1}--\ref{low:4}, $G$ contains no red copy of $T$ and no blue copy of $\cT$. Together, we get the following general lower bound on $R(T,\cT)$.
\begin{proposition}\label{prop:generallower}
Let $T$ be an $n$-vertex tree with bipartition class sizes $t_1\geq t_2$. Let $\cT$ be a $\nu$-vertex tree with bipartition class sizes $\tau_1\geq\tau_2$. If $n\geq\nu$, then
\begin{equation}\label{eq:generallower}
R(T,\cT)\geq\max\{n+\tau_2,\nu+\min\{t_2,\nu\},\min\{2t_1,2\nu\},2\tau_1\}-1.\tag{$\ddagger$}
\end{equation}
\end{proposition}
From now on, we will use $\underline{R}(T,\cT)$ to denote the quantity on the right hand side of~(\ref{eq:generallower}).

The main result of this paper is that if $T$ and $\cT$ satisfy certain conditions, then $R(T,\cT)=\underline{R}(T,\cT)$, or equivalently the lower bound~(\ref{eq:generallower}) in Proposition~\ref{prop:generallower} is tight. 
\begin{theorem}\label{thm:main}
There exists a constant $c>0$ such that the following holds.

Let $n\geq\nu$ be sufficiently large integers. Let $T$ be an $n$-vertex tree with $\Delta(T)\leq cn/\log n$ and bipartition class sizes $t_1\geq t_2$. Let $\cT$ be a $\nu$-vertex tree with $\Delta(\cT)\leq c\nu/\log\nu$ and bipartition class sizes $\tau_1\geq\tau_2$. If $\tau_2\geq t_2$ and $\nu\geq t_1$, then \[R(T,\cT)=\underline{R}(T,\cT)=\max\{n+\tau_2,2t_1\}-1.\]
\end{theorem}

Note that several conditions are required for Theorem~\ref{thm:main} to hold. From the double star counterexamples~\cite{GHK,NSZ} mentioned above, some maximum degree conditions are necessary, though our sublinear maximum degree condition may not be optimal. The condition $n\geq\nu$ can always be assumed without loss of generality, and we show that the other two conditions, $\tau_2\geq t_2$ and $\nu\geq t_1$ are both necessary. 

First, the following counterexample shows that if the condition $\tau_2\geq t_2$ is removed, then $R(T,\cT)$ can exceed $\underline{R}(T,\cT)$ by an arbitrary additive factor, even when $T$ and $\cT$ are only allowed to have constant maximum degrees. 
\begin{theorem}\label{thm:counterex1} 
Let $r\gg C\geq 1$ be integers. Then, there exist trees $T$ and $\cT$ such that the following hold.
\stepcounter{propcounter}
\begin{enumerate}[label = \emph{\textbf{\Alph{propcounter}\arabic{enumi}}}]
    \item\labelinthm{construct:1} $\Delta(T),\Delta(\cT)\leq3^{C+1}$.
    \item\labelinthm{construct:2} $T$ has bipartition class sizes $(3^C+1)r$ and $(3^C-1)r+(3^C-1)/2$.
    \item\labelinthm{construct:3} $\cT$ has bipartition class sizes $(3^C+1)r$ and $2r-(3^C-1)/2$.
    \item\labelinthm{construct:4} $R(T,\cT)\geq2(3^{C}+1)r+C-1=\underline{R}(T,\cT)+C$.
\end{enumerate}
\end{theorem}

The second counterexample shows that if the condition $\nu\geq t_1$ is removed, then we also can find $T$ and $\cT$ with constant maximum degrees, such that $R(T,\cT)$ exceeds $\underline{R}(T,\cT)$ by an arbitrary additive factor. 
\begin{theorem}\label{thm:counterex2}
Let $r\gg C\geq 2$ and $\rho\geq 2$ be integers. Then, there exist trees $T$ and $\cT$ such that the following hold.
\stepcounter{propcounter}
\begin{enumerate}[label = \emph{\textbf{\Alph{propcounter}\arabic{enumi}}}]
    \item\labelinthm{construct:2:1} $\Delta(T)\leq\rho+3$ and $\Delta(\cT)\leq 3^{C+1}+2$.
    \item\labelinthm{construct:2:2} $T$ is an $n$-vertex tree with bipartition class sizes $t_1=3^C\rho r+O_C(1)$ and $t_2=3^Cr+O_C(1)$.
    \item\labelinthm{construct:2:3} $\cT$ is a $\nu$-vertex tree with bipartition class sizes $\tau_1=((3^C-3)\rho r+(3^C+3)r)/2+O_C(1)$ and $\tau_2=3(\rho-1)r-O_C(1)$, such that $2\nu=n+\tau_2<2t_1$.
    \item\labelinthm{construct:2:4} $R(T,\cT)\geq2\nu+C-2=\underline{R}(T,\cT)+C-1$.
\end{enumerate}
\end{theorem}

Instead of proving Theorem~\ref{thm:main} directly, we reduce it to the following version with some additional conditions. 
\begin{theorem}\label{thm:main2}
There exists a constant $c>0$ such that the following holds.

Let $n\geq\nu$ be sufficiently large integers. Let $T$ be an $n$-vertex tree with $\Delta(T)\leq cn/\log n$ and bipartition class sizes $t_1\geq t_2\geq(t_1-2)/3$. Let $\cT$ be a $\nu$-vertex tree with $\Delta(\cT)\leq c\nu/\log\nu$ and bipartition class sizes $\tau_1\geq\tau_2\geq(t_1-1)/2$. If $\nu\geq t_1\geq\tau_1$, $\tau_2\geq t_2$, and $t_2+\tau_2\geq t_1-1$, then \[R(T,\cT)=\underline{R}(T,\cT)=\max\{n+\tau_2,2t_1\}-1.\]
\end{theorem}

Theorem~\ref{thm:main2} implies Theorem~\ref{thm:main} as follows. 
\begin{proof}[Proof of Theorem~\ref{thm:main}]
Let $c>0$ be the constant given by Theorem~\ref{thm:main2}, and let $T,\cT,n,\nu,t_1,t_2,\tau_1,\tau_2$ satisfy the assumptions in Theorem~\ref{thm:main}. Note that $\tau_2\geq t_2$ and $n\geq\nu$ implies that $t_1\geq\tau_1$. 

If $t_2+\tau_2\geq t_1-1$, then $2\tau_2\geq t_2+\tau_2\geq t_1-1$, so $\tau_2\geq(t_1-1)/2$. Moreover, $t_1+t_2=n\geq\nu=\tau_1+\tau_2\geq2\tau_2$, so $(t_1+t_2)/2\geq\tau_2\geq t_1-t_2-1$. This implies that $t_1+t_2\geq 2t_1-2t_2-2$, so $t_2\geq(t_1-2)/3$. Therefore, $T$ and $\cT$ satisfy the assumptions of Theorem~\ref{thm:main2}, so $R(T,\cT)=\underline{R}(T,\cT)=\max\{n+\tau_2,2t_1\}-1$.

Now suppose that $t_2+\tau_2<t_1-1$. We claim that we can find $t_1\geq t_2'\geq t_2$ and $\tau_1\geq\tau_2'\geq\tau_2$, such that $t_2'\leq\tau_2'$, $t_1+t_2'\geq\tau_1+\tau_2'$, and $t_2'+\tau_2'\in\{t_1-1,t_1\}$. If $2\tau_2\geq t_1-1$, then setting $t_2'=t_1-1-\tau_2$ and $\tau_2'=\tau_2$ works. Otherwise, $2\tau_2<t_1-1$, but $2\tau_1\geq\tau_1+\tau_2=\nu\geq t_1$ from assumption. Hence, there exists $\tau_2\leq\tau_2'\leq\tau_1$, such that $2\tau_2'\in\{t_1-1,t_1\}$. Then, setting $t_2'=\tau_2'$ works, proving the claim. 
Observe that $\max\{t_1+t_2+\tau_2,2t_1\}=2t_1=\max\{t_1+t_2'+\tau_2',2t_1\}$.

Now, by Lemma~\ref{lemma:leaves:V1}, $T$ contains a set $L$ of at least $t_1-t_2\geq t_2'-t_2$ leaves on the larger side of its bipartition. Pick $t_2'-t_2$ leaves in $L$, and attach a new leaf to each of them to obtain a new tree $T'$ with bipartition class sizes $t_1$ and $t_2'$ and maximum degree $\Delta(T')=\Delta(T)\leq cn/\log n\leq c|T'|\log |T'|$. Similarly, we can obtain a tree $\cT'$ with bipartition class sizes $\tau_1$ and $\tau_2'$ and maximum degree $\Delta(\cT')\leq c|\cT'|/\log|\cT'|$. Then, like above, $T'$ and $\cT'$ satisfy all the assumptions in Theorem~\ref{thm:main2}, so $R(T',\cT')=\max\{t_1+t_2'+\tau_2',2t_1\}-1=2t_1-1$. This means that every red/blue coloured complete graph $G$ on $2t_1-1$ vertices must contain either a red copy of $T'$ or a blue copy of $\cT'$, and thus must contain either a red copy of $T$ or a blue copy of $\cT$. Combined with the lower bound~(\ref{eq:generallower}) given by Proposition~\ref{prop:generallower}, we have $R(T,\cT)=\underline{R}(T,\cT)=2t_1-1=\max\{n+\tau_2,2t_1\}-1$, as required. 
\end{proof}

Like the proof of Theorem~\ref{thm:oldmain} in~\cite{MPY}, our proof of Theorem~\ref{thm:main2}, and thus of Theorem~\ref{thm:main}, also consists of two parts, the stability part and the extremal part. In the stability part, we use the regularity method to show that if the host graph $G$ contains no red copy of $T$ and no blue copy of $\cT$, then $G$ must approximate one of the lower bound constructions~\ref{low:1}--\ref{low:4} (see Definition~\ref{def:extremal}). This part of the proof follows the same 4-stage process as used in~\cite{MPY}, and in fact also allows the trees $T$ and $\cT$ to have up to linear maximum degrees. However, instead of using a key result of Haxell, \L uczak, Tingley~{\cite[Theorem 3]{HLT}} as a black box like in~\cite{MPY}, we modify both its proof and application, which crucially allows us to bypass the most technical part in the first stage of the proof in~\cite{MPY}. In particular, after specialising to the symmetric case when $T=\cT$, we obtained a simplified proof of~{\cite[Theorem 2.2]{MPY}}, and thus of Theorem~\ref{thm:oldmain}.

In the extremal part, we show that even if $G$ approximates one of the lower bound constructions~\ref{low:1}--\ref{low:4}, a red copy of $T$ or a blue copy of $\cT$ can still be found. This part of the proof is much simpler than the corresponding one in~\cite{MPY}, as we are not aiming for linear maximum degrees. 

The rest of the paper is organised as follows. In Section~\ref{sec:prelim}, we gather all the useful preliminary results that will be needed later. We then carry out the stability part and the extremal part of the proof of Theorem~\ref{thm:main2} in Section~\ref{sec:stability} and Section~\ref{sec:extremal}, respectively. The examples proving Theorem~\ref{thm:counterex1} and Theorem~\ref{thm:counterex2} are constructed in Section~\ref{sec:counter}. Finally, we discuss some related results and open questions in Section~\ref{sec:conclude}. 

\section{Preliminaries}\label{sec:prelim}
In this section, we gather the preliminary results that will be used in the main proofs later on. We begin with some results on concentration in Section~\ref{sec:concentration}, on matchings in Section~\ref{sec:matching}, and on trees in Section~\ref{sec:tree}. We then recall some basic facts about the regularity method in Section~\ref{sec:regularity}, before recording in Section~\ref{sec:ems} the five tree embedding methods using regularity that we borrow from~\cite{MPY}. Finally, in Section~\ref{sec:weight}, we discuss weight functions on digraphs, which is the main tool we borrow from~\cite{HLT} in order to begin the stability part of our proof.

\subsection{Concentration results}\label{sec:concentration}
We record here two well-known concentration inequalities. 
\begin{lemma}[Azuma's Inequality~{\cite[Lemma 4.2]{W}}]\label{lemma:azuma}
Let $X_1,\ldots, X_n$ be a sequence of random variables, such that for each $i\in[n]$, there exist constants $a_i\in\mathbb{R}$ and $c_i>0$ with $|X_i-a_i|\leq c_i$.
\begin{itemize}
    \item If $\mathbb{E}[X_i\mid X_1,\ldots,X_{i-1}]\geq a_i$ for every $i\in[n]$, then for every $t>0$, \[\textstyle\mathbb{P}\left(\sum_{i=1}^n(X_i-a_i)\leq-t\right)\leq\exp\left(-\frac{t^2}{2\sum_{i=1}^nc_i^2}\right).\]
    \item If $\mathbb{E}[X_i\mid X_1,\ldots,X_{i-1}]\leq a_i$ for every $i\in[n]$, then for every $t>0$, \[\textstyle \mathbb{P}\left(\sum_{i=1}^n(X_i-a_i)\geq t\right)\leq\exp\left(-\frac{t^2}{2\sum_{i=1}^nc_i^2}\right).\]
\end{itemize}
\end{lemma}

\begin{lemma}[Chernoff Bound~{\cite[Corollary 2.3, Theorem 2.10]{JLR}}]\label{lemma:chernoff}
Let $X$ be either a binomial random variable or a hypergeometric random variable. Then, for every $0<\eps\le 3/2$,
\[\mathbb{P}\left(|X-\mathbb E[X]|\ge \eps\mathbb E[X]\right)\le 2\exp(-\eps^2\mathbb{E}[X]/3).\]
\end{lemma}

\subsection{Matchings}\label{sec:matching}
In many of our later tree embedding arguments, we will first embed all but a small set of vertices with degrees 1 or 2 in $T$. To embed these low degree vertices at the end, we use a Hall type matching argument via the following well-known Hall's Theorem and its generalisation. The conditions in Lemma~\ref{lemma:Hall} and Lemma~\ref{lemma:hallmatching} will both be referred to as Hall's condition.
\begin{lemma}[Hall's Theorem {\cite[Theorem 1]{H}}]\label{lemma:Hall}
Let $G$ be a bipartite graph with bipartition classes $A$ and $B$. If $|N(I)|\ge |I|$ for every $I\subset A$, then $G$ contains a matching covering all vertices in $A$.
\end{lemma}
\begin{lemma}[{\cite[Corollary 11]{Bo}}]\label{lemma:hallmatching}
Let $G$ be a bipartite graph with bipartition classes $A$ and $B$, and let $(f_a)_{a\in A}$ be a tuple of non-negative integers indexed by elements of $A$. Suppose that $|N(I)|\geq\sum_{a\in I}f_a$ for every $I\subset A$. Then, there exists a collection of vertex-disjoint stars $(F_a)_{a\in A}$ in $G$, such that for each $a\in A$, $F_a$ is centred at $a$ and has exactly $f_a$ leaves.
\end{lemma}

We also use the following Cascading Lemma later to obtain a large complete bipartite graph in one colour whenever we are given a maximum matching in the other colour. 
\begin{lemma}[Cascading Lemma {\cite[Lemma 5.10]{MPY}}]\label{lemma:improvemaximummatching}
Let $G$ be a bipartite graph with bipartition classes $A$ and $B$, and let $M$ be a maximum matching in $G$. Let $A_M=A\cap V(M)$, $B_M=B\cap V(M)$, $A'=A\setminus A_M$, and $B'=B\setminus B_M$. Then, $A_M$ and $B_M$ can be partitioned as $A_M=A^+\cup A^-\cup \overline{A}$ and $B_M=B^+\cup B^-\cup \overline{B}$ such that the following hold.
\begin{itemize}
    \item $M$ matches vertices in $A^+$ with vertices in $B^-$, vertices in $A^-$ with vertices in $B^+$, and vertices in $\overline{A}$ with vertices in $\overline{B}$.
    \item $G[A'\cup A^-,B'\cup B^-\cup\overline{B}]$ and $G[A'\cup A^-\cup\overline{A},B'\cup B^-]$ are both empty graphs.
\end{itemize}
\end{lemma}
\subsection{Trees}\label{sec:tree}
In this subsection, we collect several elementary but useful results about trees. The first one says that every tree either has many bare paths or many leaves. In what follows, a path contained in a tree is called a \textit{bare path} if all of its internal vertices have degree exactly 2. 
\begin{lemma}[{\cite[Lemma 2.1]{K}}]\label{lemma:paths-leaf}
Let $k,\ell, n$ be positive integers. Then, every $n$-vertex tree $T$ either contains at least $\ell$ leaves, or contains a collection of at least $\frac{n}{k+1}-(2\ell-2)$ vertex-disjoint bare paths of length $k$.
\end{lemma}

The next result says that given $t_1\geq t_2$, there is always a tree with bipartition class sizes $t_1$ and $t_2$, whose maximum degree is not much bigger than $t_1/t_2$. 
\begin{lemma}\label{lemma:maxdegbound}
Let $t_1\geq t_2\geq1$ be integers. Then, there exists a tree $T$ with bipartition class sizes $t_1$ and $t_2$, and $\Delta(T)\leq\ceil{(t_1+t_2+1)/t_2}\leq t_1/t_2+2$.
\end{lemma}
\begin{proof}
Let $d=\ceil{(t_1-t_2+1)/t_2}$. Let $P$ be a path with $2t_2-1$ vertices $v_1,v_2,\ldots,v_{2t_2-1}$. Then, for each $i\in[t_2]$, attach either $d-1$ or $d$ leaves to $v_{2i-1}$, so that a total of $t_1-t_2+1$ leaves are attached. Let $T$ be the resulting tree. It is easy to see that $T$ has bipartition class sizes $t_1$ and $t_2$, and $\Delta(T)\leq d+2=\ceil{(t_1+t_2+1)/t_2}\leq t_1/t_2+2$, as required. 
\end{proof}

The next lemma says that every tree must have a certain number of leaves in its larger bipartition class. 
\begin{lemma}[{\cite[Lemma 2.10]{MPY}}]\label{lemma:leaves:V1}If $T$ is an $n$-vertex tree with bipartition classes $V_1$ and $V_2$ such that $|V_1|=t_1$, $|V_2|=t_2$, and $t_1\ge t_2$, then $T$ contains at least $t_1-t_2+1$ leaves in $V_1$.
\end{lemma}

Two subtrees $T_1$ and $T_2$ of $T$ \textit{decompose} $T$ if $T_1$ and $T_2$ share exactly one common vertex, and $V(T_1)\cup V(T_2)=V(T)$. The next two results are about finding useful decompositions of $T$. 
\begin{lemma}[{\cite[Corollary 2.14]{MPY}}]\label{lemma:splittree} 
Let $T$ be an $n$-vertex tree. Then, there exist subtrees $T_1$ and $T_2$ decomposing $T$, such that $\ceil{n/3}\leq|T_1|\leq|T_2|\leq\ceil{2n/3}$.
\end{lemma}

\begin{lemma}[{\cite[Proposition 7.4]{MPY}}]\label{lemma:bipartitesplittree}
Let $1/n\ll\mu\ll1$, and let $T$ be an $n$-vertex tree with bipartition classes $V_1$ and $V_2$ such that $|V_1|\geq 1.1|V_2|$. Then, there exist subtrees $T_1$ and $T_2$ decomposing $T$, such that $10\mu n\leq|V(T_1)\cap V_1|-|V(T_1)\cap V_2|\leq 25\mu n$.
\end{lemma}

The final result in this subsection finds in every tree a cut with some useful properties. 
\begin{lemma}[{\cite[Proposition 7.2]{MPY}}]\label{lemma:splittreewith2vertices} 
Let $1/n\ll \eps\ll 1$. Let $T$ be an $n$-vertex tree.
Then, there is a partition $V(T)=A\cup B$ with $|A|,|B|\leq (2/3-\eps)n$ such that $T[A]$ is a tree and $\{v\in A:d_T(v,B)>0\}$ is an independent set in $T$ with size at most 2.
\end{lemma}

\subsection{Szemer\'edi's Regularity Lemma}\label{sec:regularity}
In this subsection, we give an overview on Szemer\'edi's Regularity Lemma, in particular the coloured version that we will use. 

Let $G$ be a bipartite graph with bipartition classes $A$ and $B$. For sets $X\subset A$ and $Y\subset B$, the \textit{density} between $X$ and $Y$ is defined as
\[d(X,Y)=\frac{e(X,Y)}{|X||Y|}.\]
We say $G$ is \textit{$\varepsilon$-regular} if for every $X\subset A$ and every $Y\subset B$ with $|X|\ge \varepsilon |A|$ and $|Y|\ge \varepsilon |B|$, we have $|d(X,Y)-d(A,B)|\le\varepsilon$. Furthermore, we say $G$ is \textit{$(\varepsilon,d)$-regular} if $G$ is $\varepsilon$-regular and $d(A,B)\ge d$. 
The following well-known lemma shows that regularity is somewhat preserved after taking subsets.
\begin{lemma}\label{lemma:regularity:1}
Let $\eps\leq1/4$, and let $G$ be a bipartite graph with bipartition classes $A$ and $B$ that is $(\varepsilon,d)$-regular. Suppose $X\subset A$ and $Y\subset B$ satisfy $|X|\ge\sqrt\varepsilon |A|$ and $|Y|\ge\sqrt\varepsilon|B|$, then $G[X,Y]$ is $(\sqrt\varepsilon,d-\eps)$-regular.
\end{lemma}




The starting point of the stability part of our proof is the following colourful variant of Szemer\'edi's Regularity Lemma.
\begin{theorem}[Coloured Regularity Lemma{~\cite[Theorem~1.18]{komlos1995szemeredi}}]\label{theorem:regularity}
Let $1/k_2\ll 1/k_1\ll \varepsilon$. Every red/blue coloured graph $G$ on $n\ge k_1$ vertices contains disjoint subsets $V_1,\ldots,V_k\subset V(G)$ with $k_1\le k\le k_2$ that satisfy the following.
\begin{itemize}
    \item\label{colreg:1} $|V(G)\setminus(V_1\cup\cdots\cup V_k)|\le \varepsilon n$.
    \item\label{colreg:2} $|V_1|=\cdots=|V_k|$.
    \item\label{colreg:3} For all but at most $\varepsilon k^2$ indices $1\le i<j\le k$, both $\red{G}[V_i,V_j]$ and $\blue{G}[V_i,V_j]$ are $\varepsilon$-regular.
\end{itemize}
\end{theorem}

For technical reasons, we sometimes require the sets $V_i$ to have different sizes, but do not necessarily need them to cover all but $\eps n$ vertices in $G$. As this is a minor point, we do not introduce more notation and instead use the standard term \textit{$\eps$-regular partition} under the following more relaxed definition. 
\begin{definition}\label{def:regpartition}
Let $1/n\ll\eps\ll d\leq1$, and let $G$ be a red/blue coloured graph on $n$ vertices. An \textit{$\eps$-regular partition} in $G$ is a collection of disjoint subsets $V_1,\ldots,V_k\subset V(G)$, such that for all but at most $\varepsilon k^2$ pairs of indices $1\le i<j\le k$, both $\red{G}[V_i,V_j]$ and $\blue{G}[V_i,V_j]$ are $\varepsilon$-regular. Each set $V_i$ in an $\eps$-regular partition is called a \textit{cluster}. 

Given an $\eps$-regular partition $V_1\cup\cdots\cup V_k$ in $G$, its corresponding \textit{$(\eps,d)$-reduced graph} $H$ is a red/blue coloured graph with vertex set $[k]$, such that for each $\ast\in\{\text{red},\text{blue}\}$ and any distinct $i,j\in[k]$, there is an $ij$ edge of colour $\ast$ in $H$ if and only if $G_\ast[V_i,V_j]$ is $(\eps,d)$-regular. 
\end{definition}

Note that if $V_1,\ldots,V_k$ form an $\eps$-regular partition in a red/blue coloured complete graph and $\eps\ll d\leq 1/2$, then for all but at most $\eps k^2$ pairs of indices $1\leq i<j\leq k$, there is either a red edge $ij$ or a blue edge $ij$ (or both) in the corresponding $(\eps,d)$-reduced graph $H$.

Finally, we record the following refinement result that will be used later. In what follows, a graph $H$ is \textit{$\mu$-almost complete} if $d(v)\geq(1-\mu)|H|$ for every $v\in V(H)$, and a bipartite graph $H$ with bipartition classes $U_1$ and $U_2$ is \textit{$\mu$-almost complete} if $d(v,U_{3-i})\geq(1-\mu)|U_{3-i}|$ for every $i\in[2]$ and every $v\in U_i$. 
\begin{lemma}[{\cite[Lemma 2.24]{MPY}}]\label{lemma:regularity:refine}
Let $1/k,1/m\ll\eps\ll\eta\ll\alpha\ll d\leq1$. Suppose $G$ is a graph containing disjoint subsets $V_1,\ldots,V_k\subset V(G)$, each of size $m$. Let $H$ be a graph on $[k]$ such that for every $ij\in E(H)$,  $G[V_i,V_j]$ is $(\eps,d)$-regular. Suppose there exists a partition $[k]=I_1\cup I_2$, with $|I_1|=k_1,|I_2|=k_2$, and $k_1,k_2\geq\alpha k$, such that $H[I_1,I_2]$ is $\eta$-almost complete. Then, there exist two collections of disjoint sets $\{U_i:i\in J\}$ and $\{W_i:i\in J\}$ such that the following hold.
\begin{itemize}
    \item $|U_i|=|U_j|$ and $|W_i|=|W_j|$ for any $i,j\in J$.
    \item $\sum_{i\in J}|U_i|\geq(1-\alpha)\sum_{i\in I_1}|V_i|$, $\sum_{i\in J}|W_i|\geq(1-\alpha)\sum_{i\in I_2}|V_i|$.
    \item $G[U_i,W_i]$ is $(\sqrt\eps,d-\eps)$-regular for every $i\in J$.
\end{itemize}
\end{lemma}

\subsection{Embedding methods with regularity}\label{sec:ems}
In this subsection, we gather five key embedding lemmas proved with regularity in~\cite{MPY}. Each of these says that given a suitable structure in the reduced graph $H$ of a red/blue coloured complete graph $G$, we can embed a tree $T$ using regularity. 

In~\cite{MPY}, each of these methods is proved under the assumption $t_2\leq t_1\leq 2t_2$, but we will use the weaker assumption $t_2\leq t_1\leq3.1t_2$ instead, as that is what we need later. It turns out that no other changes in the statements are needed, as the original size constraints are quite generous. We also omit the proofs as the corresponding proofs in~\cite{MPY} all carry through with no change needed, except for Lemma~\ref{lemma:em2a} (\textbf{EMa}) where we need to change the choice of an assignment probability from $70\alpha$ to $90\alpha$.

\begin{lemma}[\textbf{H\L T} {\cite[Lemma 4.5]{MPY}}]\label{lemma:hlt}
Let $1/n\ll c\ll 1/k\ll\varepsilon\ll\alpha\ll d\ll1$. Let $T$ be an $n$-vertex tree with $\Delta(T)\leq cn$ and bipartition classes of sizes $t_1$ and $t_2$ satisfying $t_2\leq t_1\leq3.1t_2$. Let $G$ be a graph on at most $2n$ vertices with a partition $V(G)=V_1\cup\cdots\cup V_{2k+1}$. Let $R_{\text{\emph{\bfseries H\L T}}}$ be a graph with vertex set $[2k+1]$, such that if $ij\in E(R_{\text{\emph{\bfseries H\L T}}})$ then $G[V_i,V_j]$ is $(\eps,d)$-regular. Let $i\in [2k+1]$ and suppose there is a partition $[2k+1]\setminus \{i\}=I_A\cup I_B$, with $|I_A|=|I_B|=k$, such that the following hold for some $m_A,m_B$ (see Figure~\ref{fig:HLT}). 
\stepcounter{propcounter}
\begin{enumerate}[label = \emph{\textbf{\Alph{propcounter}\arabic{enumi}}}]
    \item $|V_a|=m_A$ for each $a\in I_A$, $|V_b|=m_B$ for each $b\in I_B$, and $|V_i|\geq n/10k$.\labelinthm{prop:hlt:1}
    \item $km_A\geq t_2+\alpha n$ and $km_B\geq t_1+\alpha n$.\labelinthm{prop:hlt:2}
    \item In $R_{\text{\emph{\bfseries H\L T}}}$, $i$ is adjacent to each vertex in $I_A$, and there is a perfect matching $M$ between $I_A$ and $I_B$.\labelinthm{prop:hlt:3}
\end{enumerate}
Then, $G$ contains a copy of $T$.
\end{lemma}

\begin{lemma}[\textbf{EMa} {\cite[Lemma 4.10]{MPY}}]\label{lemma:em2a}
Let $1/n\ll c\ll 1/k\ll\eps\ll \eta \ll\alpha\ll d\ll1$. Let $T$ be an $n$-vertex tree with $\Delta(T)\le cn$ and bipartition classes of sizes $t_1$ and $t_2$ with $t_2\leq t_1\leq3.1t_2$. Let $G$ be a graph with at most $2n$ vertices and a vertex partition $V_1\cup\cdots\cup V_k$ such that $|V_1|=|V_2|=\cdots=|V_k|=m$. Let $R_{\text{\emph{\bfseries EMa}}}$ be a graph with vertex set $[k]$, such that if $ij\in E(R_{\text{\emph{\bfseries EMa}}})$ then $G[V_i,V_j]$ is $(\eps,d)$-regular.
Suppose there is a partition $[k]=I_A\cup I_B\cup I_C$ such that the following properties hold (see Figure~\ref{fig:Ema}).
\stepcounter{propcounter}
\begin{enumerate}[label = \emph{\textbf{\Alph{propcounter}\arabic{enumi}}}]
\item\labelinthm{lemma:reg:em2a:1} $|\cup_{i\in I_A}V_i|\geq t_2+200\alpha n, |\cup_{i\in I_B}V_i|\geq t_2-\alpha n, |\cup_{i\in I_A\cup I_B}V_i|=n-\alpha n$, and $|\cup_{i\in I_C}V_i|=100\alpha n$.
    \item\labelinthm{lemma:reg:em2a:2} $R_{\text{\emph{\bfseries EMa}}}[I_A,I_B]$ is an $\eta$-almost complete bipartite graph.
    \item\labelinthm{lemma:reg:em2a:3} $R_{\text{\emph{\bfseries EMa}}}[I_A,I_C]$ contains a matching $M$ covering $I_C$.
    \item\labelinthm{lemma:reg:em2a:4} $E(R_{\text{\emph{\bfseries EMa}}}[I_A\setminus V(M)])\neq\varnothing$.
\end{enumerate}
Then, $G$ contains a copy of $T$.
\end{lemma}

\begin{figure}[h]
    \centering
    \input{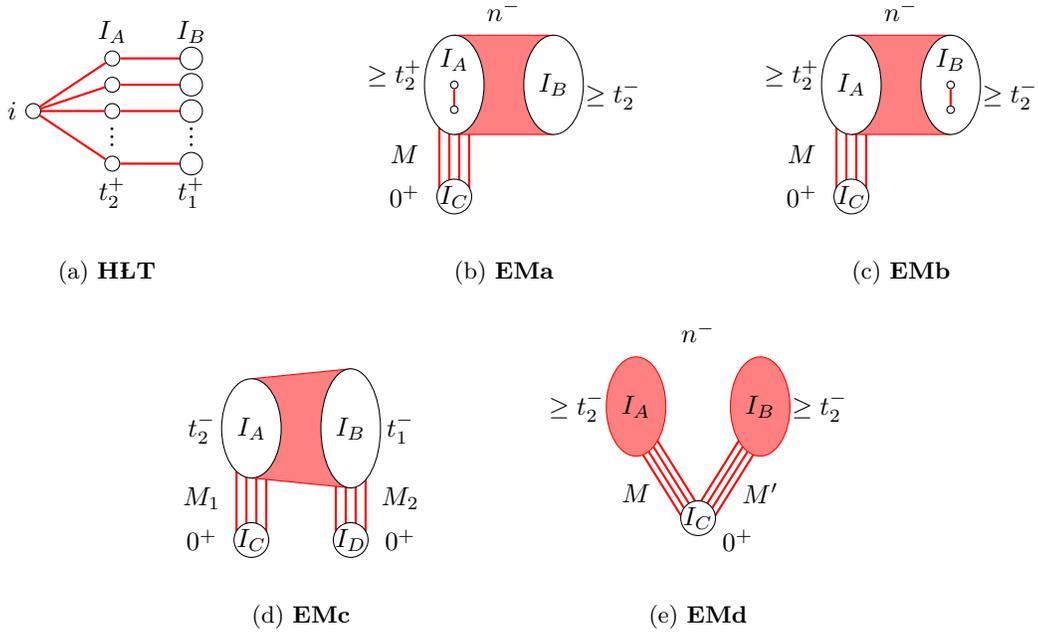}
    \caption{The reduced graph structures used in the embedding methods \textbf{H\L T} and \textbf{EMa}-\textbf{d}.}
    \label{fig:embeddings:1}
\end{figure}

\begin{lemma}[\textbf{EMb} {\cite[Lemma 4.11]{MPY}}]\label{lemma:em2b}
Let $1/n\ll 1/m\ll c\ll1/k\ll\eps\ll\eta\ll\alpha\ll d\ll1$. Let $T$ be an $n$-vertex tree with $\Delta(T)\le cn$ and bipartition classes of sizes $t_1$ and $t_2$ with $t_2\leq t_1\leq3.1t_2$. Let $G$ be a graph with at most $2n$ vertices and a vertex partition $V_1\cup\cdots\cup V_k$ with $|V_1|=\cdots=|V_k|=m$. Let $R_{\text{\emph{\bfseries EMb}}}$ be a graph with vertex set $[k]$, such that if $ij\in E(R_{\text{\emph{\bfseries EMb}}})$ then $G[V_i,V_j]$ is $(\eps,d)$-regular.
Suppose there is a partition $[k]=I_A\cup I_B\cup I_C$ such that the following properties hold (see Figure~\ref{fig:Emb}).
\stepcounter{propcounter}
\begin{enumerate}[label = \emph{\textbf{\Alph{propcounter}\arabic{enumi}}}]
    \item\labelinthm{lemma:reg:em2b:1} $|\cup_{i\in I_A}V_i|\geq t_2+200\alpha n, |\cup_{i\in I_B}V_i|\geq t_2-\alpha n, |\cup_{i\in I_A\cup I_B}V_i|=(1-\alpha)n$, and $|\cup_{i\in I_C}V_i|=100\alpha n$.
    \item\labelinthm{lemma:reg:em2b:2} $R_{\text{\emph{\bfseries EMb}}}[I_A,I_B]$ is an $\eta$-almost complete bipartite graph.
    \item\labelinthm{lemma:reg:em2b:3} $R_{\text{\emph{\bfseries EMb}}}[I_A,I_C]$ contains a matching $M$ covering $I_C$.
    \item\labelinthm{lemma:reg:em2b:4} $E(R_{\text{\emph{\bfseries EMb}}}[I_B])\neq \varnothing$.
\end{enumerate}
Then, $G$ contains a copy of $T$.
\end{lemma}


\begin{lemma}[\textbf{EMc} {\cite[Lemma 4.12]{MPY}}]\label{lemma:em2c}
Let $1/n\ll 1/m\ll c\ll1/k\ll\eps\ll\eta\ll\alpha\ll d\ll1$. Let $T$ be an $n$-vertex tree with $\Delta(T)\le cn$ and bipartition classes of sizes $t_1$ and $t_2$ with $t_2\leq t_1\leq3.1t_2$. Let $G$ be a graph on at most $2n$ vertices with a partition $V(G)=V_1\cup\cdots\cup V_k$ satisfying $|V_1|=\cdots=|V_k|=m$. Let  $R_{\text{\emph{\bfseries EMc}}}$ be a graph with vertex set $[k]$, such that if $ij\in E(R_{\text{\emph{\bfseries EMc}}})$ then $G[V_i,V_j]$ is $(\eps,d)$-regular. Suppose there is a partition $[k]=I_A\cup I_B\cup I_C\cup I_D$ such that the following properties hold (see Figure~\ref{fig:Emc}).
\stepcounter{propcounter}
\begin{enumerate}[label = \emph{\textbf{\Alph{propcounter}\arabic{enumi}}}]
    \item\labelinthm{lemma:em2c:1} $|\cup_{i\in I_A}V_i|\geq(1-\alpha) t_2, |\cup_{i\in I_B}V_i|\geq(1-11\alpha)t_1$, $|\cup_{i\in I_C}V_i|=100\alpha t_2$, and $|\cup_{i\in I_D}V_i|=10\alpha t_2$.
    \item\labelinthm{lemma:em2c:2} $R_{\text{\emph{\bfseries EMc}}}[I_A,I_B]$ is an $\eta$-almost complete bipartite graph.
    \item\labelinthm{lemma:em2c:3} $R_{\text{\emph{\bfseries EMc}}}[I_A,I_C]$ contains a matching $M_1$ covering $I_C$.
    \item\labelinthm{lemma:em2c:4} $R_{\text{\emph{\bfseries EMc}}}[I_B,I_D]$ contains a matching $M_2$ covering $I_D$.
\end{enumerate}
Then, $G$ contains a copy of $T$.
\end{lemma}


\begin{lemma}[\textbf{EMd} {\cite[Lemma 4.13]{MPY}}]\label{lemma:em3}
Let $1/n\ll 1/m\ll c\ll1/k\ll\eps\ll\eta\ll\alpha\ll d\ll1$. Let $T$ be an $n$-vertex tree with $\Delta(T)\le cn$ and bipartition classes of sizes $t_1$ and $t_2$ with $t_2\leq t_1\leq3.1t_2$. Let $G$ be a graph with a partition $V(G)=V_1\cup\cdots\cup V_k$ satisfying $|V_1|=\cdots=|V_k|=m$. Let  $R_{\text{\emph{\bfseries EMd}}}$ be a graph with vertex set $[k]$, such that if $ij\in E(R_{\text{\emph{\bfseries EMd}}})$ then $G[V_i,V_j]$ is $(\eps,d)$-regular. Suppose there is a partition $[k]=I_A\cup I_B\cup I_C$ such that the following properties hold (see Figure~\ref{fig:Emd}).
\stepcounter{propcounter}
\begin{enumerate}[label = \emph{\textbf{\Alph{propcounter}\arabic{enumi}}}]
    \item\labelinthm{lemma:em2d:1} $|\cup_{i\in I_A}V_i|\geq|\cup_{i\in I_B}V_i|\geq(1-\alpha) t_2, |\cup_{i\in I_A\cup I_B}V_i|\geq(1-\alpha)n$, and $|\cup_{i\in I_A\cup I_B\cup I_C}V_i|=(1+100\alpha)n$.
    \item\labelinthm{lemma:em2d:2} $R_{\text{\emph{\bfseries EMd}}}[I_A]$ and $R_{\text{\emph{\bfseries EMd}}}[I_B]$ are both $\eta$-almost complete graphs.
    \item\labelinthm{lemma:em2d:3} $R_{\text{\emph{\bfseries EMd}}}[I_A,I_C]$ contains a matching $M$ covering $I_C$, and a vertex in $I_A\cap V(M)$ that is adjacent to every vertex in $I_C\cap V(M)$.
    \item\labelinthm{lemma:em2d:4} $R_{\text{\emph{\bfseries EMd}}}[I_B,I_C]$ contains a matching $M'$ covering $I_C$.
\end{enumerate}
Then, $G$ contains a copy of $T$.
\end{lemma}

\subsection{Weight functions on digraphs}\label{sec:weight}
In this subsection, we define weight functions on digraphs, and gather several properties that they satisfy. This is the main tool used by Haxell, \L uczak, and Tingley to prove a key result~{\cite[Theorem 3]{HLT}}, which is then used by Montgomery, Pavez-Sign{\'e}, and Yan as a black box in~\cite{MPY} to prove the stability part of their main result. In the stability part of our proof, we use this weight function tool to modify both the proof and application of~{\cite[Theorem 3]{HLT}}, in order to bypass the most technical part of the proof in~\cite{MPY}. 
\begin{definition}
Let $\vv{H}=(V,E)$ be a digraph, let $v\in V$, and let $\alpha\geq1$.
\begin{itemize}
    \item $D^+(v)$ denotes the set of all arcs of $\vv{H}$ pointing out from $v$, and $D^-(v)$ denotes the set of all arcs of $\vv{H}$ pointing into $v$.
    \item A non-negative function $f:E\to\mathbb{R}$ is an $\alpha$-\emph{weight function} on $\vv{H}$ if for all $v\in V$,\[w_\alpha(v,f):=\frac1\alpha\sum_{e\in D^+(v)}f(e)+\sum_{e\in D^-(v)}f(e)\leq 1.\] 
    \item Denote the set of all $\alpha$-weight functions on $\vv{H}$ by $\mathcal{F}_\alpha(\vv{H})$, or simply $\mathcal{F}_\alpha$ when the choice of $\vv{H}$ is clear.
    \item For each $f\in\mathcal{F}_\alpha$, define $w(f)=\sum_{e\in E}f(e)$.
    \item Let $w_{\alpha,\max}=\max\{w(f):f\in\mathcal{F}_\alpha\}$, and let $\mathcal{F}_{\alpha,\max}=\{f\in\mathcal{F}_\alpha:w(f)=w_{\alpha,\max}\}$. By compactness, $w_{\alpha,\max}$ and $\mathcal{F}_{\alpha,\max}$ are both well-defined. The $\alpha$ subscript may be omitted for brevity if the choice of $\alpha$ is clear.
    \end{itemize} 
\end{definition}

\begin{definition}
Let $H=(V,E)$ be a red/blue coloured graph, and let $x\in V$. For every $\ast\in\{\text{red},\text{blue}\}$, let $\vv{H}_\ast(x)$ denote the digraph with vertex set $V\setminus\{x\}$, such that $yz$ is an arc of $\vv{H}_\ast(x)$ if and only if both $xy$ and $yz$ are edges in $H_\ast$. 
\end{definition}

Given a weight function with maximum weight on a (coloured) digraph $\vv{H}$, the following two results provide decompositions of the vertex set of $\vv{H}$ with useful properties.
\begin{lemma}[{\cite[Theorem 8]{HLT}}]\label{lemma:ABC}
Let $\vv{H}=(V,E)$ be a digraph, and let $\alpha>1$ be fixed. Let
\begin{align*}
A&=\{a\in V: w(a,f)<1 \text{ for some } f\in \mathcal{F}_{\max}\},\\
B&=\{b\in V: ba\in E \text{ for some } a\in A\},\\
C&=V\setminus(A\cup B).
\end{align*}
Then the following hold.
\stepcounter{propcounter}
\begin{enumerate}[label = \emph{\textbf{\Alph{propcounter}\arabic{enumi}}}]
    \item\labelinthm{ABC:5} $|B|\le w_{\max}/\alpha$.
    \item\labelinthm{ABC:6} $|C|\le (1+1/\alpha)w_{\max}-(1+\alpha)|B|$.
    \item\labelinthm{ABC:8} $|A|>|V|-(1+1/\alpha)w_{\max}$.
\end{enumerate}
\end{lemma}

\begin{lemma}[{\cite[Lemma 13]{HLT}}]\label{lemma:QR}
Let $H=(V,E)$ be a red/blue coloured graph, and let $x\in V$. For any $\ast\in\{\textup{red},\textup{blue}\}$, let $Q=N_\ast(x)$ and $R=V\setminus(Q\cup\{x\})$. Let $A,B,C$ be the partition of $V\setminus\{x\}$ obtained by applying Lemma~\ref{lemma:ABC} to $\vv{H}_\ast(x)$ with some fixed $\alpha>1$. Then, the following hold.
\stepcounter{propcounter}
\begin{enumerate}[label = \emph{\textbf{\Alph{propcounter}\arabic{enumi}}}]
    \item\labelinthm{QR:1} $B\cap R=\varnothing$.
    \item\labelinthm{QR:2} All edges within $Q\cap A$, and all edges between $Q\cap A$, $Q\cap C$, and $R\cap A$ are not colour $\ast$. 
    \item\labelinthm{QR:3} $|Q\cap B|+|R\cap C|\leq w_{\ast,\max}$.
\end{enumerate}
\end{lemma}

The next lemma says that if a red/blue coloured almost complete graph $H$ contains a vertex $v$ with high red degree, but no weight function on $\red{\vv{H}}(v)$ has large enough weight, then $H$ contains a large almost complete bipartite graph in blue.  
\begin{lemma}\label{lem:getBsituation}
Let $\eps\ll1<\beta\leq\alpha$. Then, for all sufficiently large $k$ the following holds. 

Let $H=(V,E)$ be a red/blue coloured $\eps$-almost complete graph on $(1+1/\alpha+1/\beta-2\eps)k$ vertices, and let $v\in V$. Suppose that $\red{d}(v)\geq(1/\alpha+1/\beta-5\eps)k$. Let $A,B,C$ be obtained by applying Lemma~\ref{lemma:ABC} to $\red{\vv{H}}(v)$ and $\alpha$, let $Q=\red{N}(v)$ and $R=V\setminus(Q\cup\{v\})$. 

Suppose that $w_{\textup{red},\max}(\red{\vv{H}}(v))<(1+\eps)k$ and $|Q\cap A|\geq1$. Then, there exist $z\in V\setminus\{v\}$ and disjoint $X,Y\subset V\setminus\{v,z\}$, such that $|X|,|Y|\geq(1/\alpha-10\eps)k$, $|X|+|Y|\geq(1/\alpha+1/\beta-10\eps)k$, $z$ is a blue neighbour of all vertices in $X\cup Y$, and all edges between $X$ and $Y$ in $H$ are blue.
\end{lemma}
\begin{proof}
By~\ref{QR:1} and~\ref{QR:3}, $|Q\cap A|+|Q\cap C|+|R\cap A|\geq(1+1/\alpha+1/\beta-3\eps)k-(1+\eps)k>(1/\alpha+1/\beta-4\eps)k$. By~\ref{QR:2}, all edge within $Q\cap A$, and all edges between $Q\cap A$, $Q\cap C$, and $R\cap A$ are blue. By~\ref{ABC:5}, $|Q\cap B|\leq|B|\leq(1+\eps)k/\alpha<(1/\alpha+\eps)k$, so $|Q\cap A|+|Q\cap C|\geq(1/\beta-6\eps)k$. By~\ref{ABC:8}, $|Q\cap A|+|R\cap A|=|A|>|V|-1-(1+1/\alpha)(1+\eps)k>(1/\beta-5\eps)k$. 

Let $z$ be any vertex in $Q\cap A$, and note that from assumptions, $z$ is a blue neighbour of all but at most $3\eps k$ vertices in $(Q\cap A)\cup(Q\cap C)\cup(R\cap A)$. If $|Q\cap C|\geq(1/\alpha-6\eps)k$, then setting $X=\blue{N}(z,Q\cap C)$ and $Y=\blue{N}(z,(Q\cap A)\cup(R\cap A))$ works. Otherwise, there exists a subset $Q\cap C\subset X\subset(Q\cap C)\cup(Q\cap A)$ with $|X|=(1/\alpha-10\eps)k$ and $X\subset\blue{N}(z)$. Then, setting $Y=\blue{N}(z,(Q\cap A)\cup(R\cap A))\setminus X$ works.
\end{proof}


Finally, the reason we are interested in weight functions is the following lemma, which shows that given an $\alpha$-weight function with large weight in the reduced graph, we can find the structure required to apply Lemma~\ref{lemma:hlt} (\textbf{H\L T}) to embed a tree whose bipartition class sizes $t_1$ and $t_2$ have ratio $t_1/t_2=\alpha$. It follows from the derivation in~{\cite[Section 8]{HLT}} of~{\cite[Theorem 3]{HLT}} from~{\cite[Theorem 10]{HLT}}, after modifying some constants.
\begin{lemma}\label{lem:getHLT}
Let $1/t_1\ll1/k_2\ll1/k_1\ll\eps_1\ll\eps_2\ll\zeta\ll d\leq1\leq\alpha$. Let $G$ be a graph with at least $(1+1/\alpha+10\zeta)t_1$ vertices. Let $V_1\cup\cdots\cup V_{k_1}$ be an $\eps_1$-regular partition in $V(G)$ containing at least $(1+1/\alpha+5\zeta)t_1$ vertices in total, with $|V_1|=\cdots=|V_{k_1}|=n'$. Let $H=(V,E)$ be the corresponding $(\eps_1,2d)$-reduced graph with vertex set $V=[k]$.

If for some $x\in V$, there exists an $\alpha$-weight function $f$ on $\vv{H}(x)$ with $w(f)\geq(1+2\zeta)t_1/n'$, then there exists an $\eps_2$-regular partition $U_0\cup\cdots\cup U_{k_2}$ in $G$ with corresponding $(\eps_2,d)$-reduced graph $H'$, and a partition $[k_2]=I_A\cup I_B\cup I_C$, such that the following hold.
\stepcounter{propcounter}
\begin{enumerate}[label = \emph{\textbf{\Alph{propcounter}\arabic{enumi}}}]
    \item $|U_i|=m$ for every $i\in\{0\}\cup I_A$ and $|U_i|=\alpha m$ for every $i\in I_B$.
    \item $|I_A|=|I_B|$, with $|I_B|\alpha m\geq(1+\zeta)t_1$.
    \item In $H'$, 0 is adjacent to every $a\in I_A$. 
    \item $H'[I_A,I_B]$ contains a perfect matching.
\end{enumerate}
\end{lemma}

\section{Stability part}\label{sec:stability}
In this section, we prove Theorem~\ref{thm:stability}, which says that either we can find a red copy of $T$ or a blue copy of $\cT$ using regularity methods, or the graph $G$ must be extremal, in the sense that it approximates one of the lower bound constructions~\ref{low:1}--\ref{low:4}. After specialising to the case when $T=\cT$, this simplifies the proof of the corresponding stability result in{~\cite[Theorem 2.2]{MPY}}, and leads to an easier proof of Theorem~\ref{thm:oldmain}. Before stating Theorem~\ref{thm:stability}, we first need to formalise what we mean by being extremal, as follows. 
\begin{definition}\label{def:extremal}
Let $0<\mu<1$. Let $T$ be an $n$-vertex tree with bipartition class sizes $t_1\geq t_2$, let $\cT$ be a $\nu$-vertex tree with bipartition class sizes $\tau_1\geq\tau_2$, and suppose that $\nu\geq t_1$ and $\tau_2\geq t_2$. Let $G$ be a red/blue coloured complete graph. 
\begin{itemize}
    \item $G$ is \textit{Type 1 $(\mu,T,\cT)$-extremal} if there are disjoint subsets $U_1,U_2\subset V(G)$ such that (see~\ref{low:1}) 
    \begin{itemize}
        \item $|U_1|\geq(1-\mu)n$ and $|U_2|\geq(1-\mu)\tau_2$,
        \item for every $u\in U_1$, $\blue{d}(u,U_1)\leq \mu n$, and
        \item for every $i\in[2]$ and every $u\in U_i$, $\red{d}(u,U_{3-i})\leq \mu n$.
    \end{itemize}

\item $G$ is \textit{Type 2 $(\mu,T,\cT)$-extremal} if there are disjoint subsets $U_1,U_2\subset V(G)$ such that (see~\ref{low:2}) 
    \begin{itemize}
        \item $|U_1|\geq(1-\mu)\nu$ and $|U_2|\geq(1-\mu)t_2$,
        \item for every $u\in U_1$, $\red{d}(u,U_1)\leq \mu n$, and
        \item for every $i\in[2]$ and every $u\in U_i$, $\blue{d}(u,U_{3-i})\leq \mu n$.
    \end{itemize}

\item $G$ is \textit{Type 3 $(\mu,T,\cT)$-extremal} if there are disjoint subsets $U_1,U_2\subset V(G)$ such that (see~\ref{low:3})
    \begin{itemize}
        \item $|U_1|,|U_2|\geq (1-\mu)t_1$, and
        \item for each $i\in [2]$ and $u\in U_i$, $\red{d}(u,U_i)\leq\mu n$ and $\blue{d}(u,U_{3-i})\leq\mu n$.
    \end{itemize}

\item $G$ is \textit{Type 4 $(\mu,T,\cT)$-extremal} if there are disjoint subsets $U_1,U_2\subset V(G)$ such that (see~\ref{low:4})
    \begin{itemize}
        \item $|U_1|,|U_2|\geq (1-\mu)\tau_1$, and
        \item for each $i\in [2]$ and $u\in U_i$, $\blue{d}(u,U_i)\leq\mu n$ and $\red{d}(u,U_{3-i})\leq\mu n$.
    \end{itemize}

\item $G$ is \emph{$(\mu,T,\cT)$-extremal} if it is Type 1, Type 2, Type 3, or Type 4 $(\mu,T,\cT)$-extremal.
\end{itemize}
\end{definition}

Now we can state the main result of this section. 
\begin{theorem}\label{thm:stability}
Let $1/n\leq1/\nu\ll c\ll\mu\ll1$. Let $T$ be an $n$-vertex tree with $\Delta(T)\leq cn$ and bipartition class sizes $t_1\geq t_2$, and let $\cT$ be a $\nu$-vertex tree with $\Delta(\cT)\leq c\nu$ and bipartition class sizes $\tau_1\geq\tau_2$. Suppose that $\tau_2\geq t_2$ and $\nu\geq t_1$. Let $G$ be a red/blue coloured complete graph with $\max\{n+\tau_2,2t_1\}-1$ vertices.
Then, at least one of the following is true. 
\begin{itemize}
    \item $G$ contains a red copy of $T$ or a blue copy of $\cT$.
    \item $G$ is $(\mu,T,\cT)$-extremal.
\end{itemize}
\end{theorem}

Similar to how Theorem~\ref{thm:main} is reduced to Theorem~\ref{thm:main2} in Section~\ref{sec:intro}, we can also reduce Theorem~\ref{thm:stability} to a version with more assumptions, Theorem~\ref{thm:stability2}, as follows.
\begin{theorem}\label{thm:stability2}
Let $1/n\leq1/\nu\ll c\ll\mu\ll1$. Let $T$ be an $n$-vertex tree with $\Delta(T)\leq cn$ and bipartition class sizes $t_1\geq t_2\geq(t_1-2)/3$. Let $\cT$ be a $\nu$-vertex tree with $\Delta(\cT)\leq c\nu$ and bipartition class sizes $\tau_1\geq\tau_2\geq(t_1-1)/2$. Suppose $\nu\geq t_1\geq\tau_1$, $\tau_2\geq t_2$, and $t_2+\tau_2\geq t_1-1$. Let $G$ be a red/blue coloured complete graph with $\max\{n+\tau_2,2t_1\}-1$ vertices.
Then, at least one of the following is true. 
\stepcounter{propcounter}
\begin{enumerate}[label = \emph{\textbf{\Alph{propcounter}\arabic{enumi}}}]
    \item\labelinthm{main:1} $G$ contains a red copy of $T$ or a blue copy of $\cT$.
    \item\labelinthm{main:2} $G$ is Type 1 $(\mu,T,\cT)$-extremal.
    \item\labelinthm{main:3} $\nu\geq(1-\mu)(t_1+\tau_2)$ and $G$ is Type 2 $(\mu,T,\cT)$-extremal.
    \item\labelinthm{main:4} $t_1\geq(1-\mu)(t_2+\tau_2)$ and $G$ is Type 3 $(\mu,T,\cT)$-extremal.
    \item\labelinthm{main:5} $\tau_1\geq(1-\mu)(n+\tau_2)/2$ and $G$ is Type 4 $(\mu,T,\cT)$-extremal.
\end{enumerate}
\end{theorem}
\begin{proof}[Proof of Theorem~\ref{thm:stability}]
Let $G,T,\cT,n,\nu,t_1,t_2,\tau_1,\tau_2,c,\mu$ satisfy the assumptions in Theorem~\ref{thm:stability}. Note that $\tau_2\geq t_2$ and $t_1+t_2=n\geq\nu=\tau_1+\tau_2$ implies that $t_1\geq\tau_1$. 

If $t_2+\tau_2\geq t_1-1$, then as in the proof of Theorem~\ref{thm:main} in Section~\ref{sec:intro}, $T$ and $\cT$ satisfy all the assumptions of Theorem~\ref{thm:stability2}, so we are done.

Now suppose that $t_2+\tau_2<t_1-1$. Again as in the proof of Theorem~\ref{thm:main} in Section~\ref{sec:intro}, there exist $t_1\geq t_2'\geq t_2$ and $\tau_1\geq\tau_2'\geq\tau_2$, such that $t_2'\leq\tau_2'$, $t_1+t_2'\geq\tau_1+\tau_2'$, and $t_2'+\tau_2'\in\{t_1-1,t_1\}$. Observe that $\max\{t_1+t_2+\tau_2,2t_1\}=2t_1=\max\{t_1+t_2'+\tau_2',2t_1\}$.

By Lemma~\ref{lemma:leaves:V1}, $T$ contains a set $L$ of at least $t_1-t_2\geq t_2'-t_2$ leaves on the larger side of its bipartition. Pick $t_2'-t_2$ leaves in $L$, and attach a new leaf to each of them to obtain a new tree $T'$ with bipartition class sizes $t_1$ and $t_2'$ and maximum degree $\Delta(T')=\Delta(T)\leq cn\leq c|T'|$. Similarly, we can obtain a tree $\cT'$ with bipartition class sizes $\tau_1$ and $\tau_2'$ and maximum degree $\Delta(\cT')\leq c|\cT'|$. Then, like above, $T'$ and $\cT'$ satisfy all the assumptions in Theorem~\ref{thm:stability2}, so $G$ contains a red copy of $T'$ and thus $T$, or a blue copy of $\cT'$ and thus $\cT$, or is $(\mu,T',\cT')$-extremal. Note that if $G$ is $(\mu,T',\cT')$-extremal, then it is also $(2\mu,T,\cT)$-extremal, so we are done.
\end{proof}

\begin{figure}[h]
\input{figures/stages/stages}
\end{figure}

Throughout this section, let $T$ and $\cT$ be fixed, and let all parameters satisfy the assumptions in Theorem~\ref{thm:stability2}. As mentioned in Section~\ref{sec:intro}, Theorem~\ref{thm:stability2} will be proved using a 4-stage process (see Figure~\ref{fig:stages}) similar to the one used in~\cite{MPY}. In each stage, we attempt to find a nice structure in the reduced graph in either red or blue, which will allow us to embed the corresponding tree using regularity. If this attempt fails, then we show that the reduced graph must contain a structure in the other colour that represents the start of the next stage. At the end of the four stages, if we managed to embed neither a red copy of $T$ nor a blue copy of $\cT$, then we can conclude that the reduced graph must be extremal (see Definition~\ref{def:redextremal}), and thus so is the original host graph $G$ (see Lemma~\ref{lemma:regtoext}). To avoid repetition, we will carry out the 4 stages in reverse order from Section~\ref{sec:stage4} to Section~\ref{sec:stage1}. Theorem~\ref{thm:stability2}, and thus Theorem~\ref{thm:stability}, will follow immediately from Lemma~\ref{lem:stage1} (\textbf{Stage 1}) in Section~\ref{sec:stage1}.

The main difference between our proof and the one in~\cite{MPY} is in \textbf{Stage 1}. Recall from Lemma~\ref{lemma:hlt} that if the reduced graph contains the \textbf{H\L T} structure (see Figure~\ref{fig:HLT}), then we can find a copy of the tree. In~\cite{MPY}, the authors use~{\cite[Theorem 3]{HLT}} as a black box and start \textbf{Stage 1} with a slightly smaller version of the \textbf{H\L T} structure, \textbf{H\L T}$^-$, which is guaranteed to exist by this key result. Instead, we start with the original red/blue coloured reduced graph, and modify the proof of~{\cite[Theorem 3]{HLT}} to show that we can either find a full-size \textbf{H\L T} structure and embed with Lemma~\ref{lemma:hlt} immediately, or find a \textbf{B-situation} or a \textbf{C-situation}, which allows us to skip ahead to \textbf{Stage 2} or \textbf{Stage 3}, respectively. This enables us to bypass the technical arguments used in \textbf{Stage 1} of~\cite{MPY}, which involves the use of three regularity embedding lemmas~{\cite[Lemma 4.6, Lemma 4.8, Lemma 4.9]{MPY}} whose proofs are very tedious.


\subsection{Extremal coloured graphs and extremal regular partitions}\label{sec:extpart}
Recall from Definition~\ref{def:extremal} what it means for a red/blue coloured complete graph $G$ to be $(\mu,T,\cT)$-extremal. In this subsection, we define a similar notion for the reduced graph of a regular partition, and show that the reduced graph being extremal implies that the host graph is extremal as well. 

\begin{definition}\label{def:redextremal}
Let $1/n\ll1/k\ll\eps\ll\mu,d\leq1$. Let $T$ be an $n$-vertex tree with bipartition class sizes $t_1\geq t_2$, let $\cT$ be a $\nu$-vertex tree with bipartition class sizes $\tau_1\geq\tau_2$, and suppose that $\nu\geq t_1$ and $\tau_2\geq t_2$. Let $G$ be a red/blue coloured complete graph. 
\begin{itemize}
\item $G$ has a \textit{Type 1 $(\mu,T,\cT)$-extremal $(\eps,d)$-regular partition} if there exists an $\eps$-regular partition $V_1\cup\cdots\cup V_k$ in $V(G)$ with corresponding $(\eps,d)$-reduced graph $R$, and a partition $[k]=I_A\cup I_B$ such that the following hold.
    \begin{itemize}
        \item $|V_i|=m$ for every $i\in[k]$.
        \item $|I_A|\geq(1-\mu)n/m$ and $|I_B|\geq(1-\mu)\tau_2/m$.
        \item Both $\blue{R}[I_A]$ and $\red{R}[I_A,I_B]$ are $\mu$-almost empty.
    \end{itemize}

\item $G$ has a \textit{Type 2 $(\mu,T,\cT)$-extremal $(\eps,d)$-regular partition} if there exists an $\eps$-regular partition $V_1\cup\cdots\cup V_k$ in $V(G)$ with corresponding $(\eps,d)$-reduced graph $R$, and a partition $[k]=I_A\cup I_B$ such that the following hold.
    \begin{itemize}
        \item $|V_i|=m$ for every $i\in[k]$.
        \item $|I_A|\geq(1-\mu)\nu/m$ and $|I_B|\geq(1-\mu)t_2/m$.
        \item Both $\red{R}[I_A]$ and $\blue{R}[I_A,I_B]$ are $\mu$-almost empty.
    \end{itemize}

\item $G$ has a \textit{Type 3 $(\mu,T,\cT)$-extremal $(\eps,d)$-regular partition} if there exists an $\eps$-regular partition $V_1\cup\cdots\cup V_k$ in $V(G)$ with corresponding $(\eps,d)$-reduced graph $R$, and a partition $[k]=I_A\cup I_B$ such that the following hold.
    \begin{itemize}
        \item $|V_i|=m$ for every $i\in[k]$.
        \item $|I_A|,|I_B|\geq(1-\mu)t_1/m$.
        \item All of $\red{R}[I_A]$, $\red{R}[I_B]$, and $\blue{R}[I_A,I_B]$ are $\mu$-almost empty.
    \end{itemize}

\item $G$ has a \textit{Type 4 $(\mu,T,\cT)$-extremal $(\eps,d)$-regular partition} if there exists an $\eps$-regular partition $V_1\cup\cdots\cup V_k$ in $V(G)$ with corresponding $(\eps,d)$-reduced graph $R$, and a partition $[k]=I_A\cup I_B$ such that the following hold.
    \begin{itemize}
        \item $|V_i|=m$ for every $i\in[k]$.
        \item $|I_A|,|I_B|\geq(1-\mu)\tau_1/m$.
        \item All of $\blue{R}[I_A]$, $\blue{R}[I_B]$, and $\red{R}[I_A,I_B]$ are $\mu$-almost empty.
    \end{itemize}
\end{itemize}
\end{definition}

With suitable choices of constants, if a red/blue coloured complete graph $G$ has a Type $i$ $(\mu',T,\cT)$-extremal $(\eps,d)$-regular partition for some $i\in[4]$, then it is Type $i$ $(\mu,T,\cT)$-extremal, as follows.
\begin{lemma}\label{lemma:regtoext}
Let $1/n\ll1/k\ll\eps\ll\mu'\ll d\ll\mu\ll1$. Let $T$ be an $n$-vertex tree with bipartition class sizes $t_1\geq t_2$, let $\cT$ be a $\nu$-vertex tree with bipartition class sizes $\tau_1\geq\tau_2$, and suppose that $\nu\geq t_1$ and $\tau_2\geq t_2$. If $G$ is a red/blue coloured complete graph on at most $2n$ vertices that contains a Type $i$ $(\mu',T,\cT)$-extremal $(\eps,d)$-regular partition for some $i\in[4]$, then $G$ is Type $i$ $(\mu,T,\cT)$-extremal.
\end{lemma}
\begin{proof}
Suppose $G$ contains a Type 1 $(\mu',T,\cT)$-extremal $(\eps,d)$-regular partition $V_1\cup\cdots\cup V_k$ satisfying the conditions in Definition~\ref{def:redextremal}. Let $V_A=\cup_{i\in I_A}V_i$ and $V_B=\cup_{i\in I_B}V_i$. From assumptions and using $km\leq|G|\leq 2n$, the number of blue edges in $V_A$ is at most $\mu'|I_A|^2m^2+d|I_A|^2m^2+|I_A|m^2\leq\mu'k^2m^2+dk^2m^2+km^2\leq4(\mu'+d+1/k)n^2\leq5dn^2$. Therefore, by removing at most $\sqrt{5d} n$ vertices from $V_A$, we can obtain a subset $V_A'$ such that $\blue{d}(u,V_A')\leq\sqrt{5d}n$ for every $u\in V_A'$.

Similarly, the number of red edges between $V_A'$ and $V_B$ is at most $\mu'|I_A||I_B|m^2+d|I_A||I_B|m^2\leq4(\mu'+d)n^2\leq 5dn^2$. Thus, we can remove at most $\sqrt{5d} n$ vertices from each of $V_A'$ and $V_B$ to obtain subsets $U_1$ and $U_2$, respectively, such that for every $i\in[2]$ and every $u\in U_i$, $\red{d}(u,U_{3-i})\leq\sqrt{5d} n$. Since $\mu\gg\sqrt{d}$, $U_1$ and $U_2$ show that $G$ is Type 1 $(\mu,T,\cT)$-extremal.

The cases when $G$ contains a Type $i$ $(\mu',T,\cT)$-extremal $(\eps,d)$-regular partition for some $2\leq i\leq 4$ are similar and thus omitted. 
\end{proof}

\subsection{Stage 4}\label{sec:stage4}

\begin{lemma}[\textbf{Stage 4}]\label{lem:stage4}
Let $1/n,1/\nu\ll 1/m\ll c\ll1/k\ll\eps\ll\zeta\ll d\ll\mu\ll1$. Let $T$ be an $n$-vertex tree with $\Delta(T)\leq cn$ and bipartition class sizes $t_1\geq t_2$. Let $\cT$ be a $\nu$-vertex tree with $\Delta(\cT)\leq c\nu$ and bipartition class sizes $\tau_1\geq\tau_2$. Suppose $\nu\geq t_1\geq \tau_1$, $\tau_2\geq t_2$, and $t_2+\tau_2\geq t_1-1$. 

Let $G$ be a red/blue coloured complete graph that contains a coloured $\eps$-regular partition $V_1\cup\cdots\cup V_k$ with $|V_1|=\cdots=|V_k|=m$ and corresponding red/blue coloured $(\eps,d)$-reduced graph $H$. Suppose there is a partition $[k]=I_A\cup I_B$ such that either~\emph{\ref{lemma:stage4:1}} and~\emph{\ref{lemma:stage4:2}} hold (see \textbf{\emph{D-situation}} in Figure~\ref{fig:stages}) or~\emph{\ref{lemma:stage4:3}} and~\emph{\ref{lemma:stage4:4}} hold.
\stepcounter{propcounter}
\begin{enumerate}[label = \emph{\textbf{\Alph{propcounter}\arabic{enumi}}}]
    \item\labelinthm{lemma:stage4:1} $|I_A|m,|I_B|m\geq(1-\zeta)t_2$ and $|I_A|m+|I_B|m\geq(1-\zeta)(n+\tau_2)$.
    \item\labelinthm{lemma:stage4:2} $\red{H}[I_A,I_B]$ is $\zeta$-almost complete.
\end{enumerate}
\stepcounter{propcounter}
\begin{enumerate}[label = \emph{\textbf{\Alph{propcounter}\arabic{enumi}}}]
    \item\labelinthm{lemma:stage4:3} $|I_A|m,|I_B|m\geq(1-\zeta)\tau_2$ and $|I_A|m+|I_B|m\geq(1-\zeta)(n+\tau_2)$.
    \item\labelinthm{lemma:stage4:4} $\blue{H}[I_A,I_B]$ is $\zeta$-almost complete.
\end{enumerate}
Then, at least one of~\emph{\ref{main:1}}--\emph{\ref{main:5}} holds. 
\end{lemma}
\begin{proof}
We will mainly prove the case when~\ref{lemma:stage4:1} and~\ref{lemma:stage4:2} hold, and only comment briefly on the case when~\ref{lemma:stage4:3} and~\ref{lemma:stage4:4} hold, as it can be proved in a similar way after exchanging the role of $\tau_2$ and $t_2$, and red and blue. Without loss of generality, assume that $|I_A|\geq|I_B|$. Let $A=\cup _{i\in I_A}V_i$ and $B=\cup_{i\in I_B}V_i$. Note that from \ref{lemma:stage4:1}, $|A|+|B|\geq(1+200\zeta)n$ and $|A|=|I_A|m\geq(1-\zeta)(n+\tau_2)/2\geq t_2+200\zeta n$.

\medskip

\noindent\textbf{Case I.} $|B|\geq t_2+100\zeta n$. If there exists an edge $ij$ in $\red{H}[I_A]$, then we can use~\ref{lemma:stage4:2} to greedily find a perfect matching between some $I_A'\subset I_A\setminus\{i,j\}$ of size $100\zeta n/m$ and some $I_B'\subset I_B$. This allows us to apply Lemma~\ref{lemma:em2a} (\textbf{EMa}) to $I_A, I_B\setminus I_B'$, and $I_B'$ to embed $T$ in red. Similarly, we can use Lemma~\ref{lemma:em2b} (\textbf{EMb}) to embed $T$ in red if there is an edge in $\red{H}[I_B]$. Thus, we may assume that both $\red{H}[I_A]$ and $\red{H}[I_B]$ are empty. This implies that both $\blue{H}[I_A]$ and $\blue{H}[I_B]$ contain at most $\eps k^2$ non-edges, so we can find $J_A\subset I_A$ and $J_B\subset I_B$, such that $|J_A|\geq(1-10\sqrt\eps)|I_A|$, $|J_B|\geq(1-10\sqrt\eps)|I_B|$, and both $\blue{H}[J_A], \blue{H}[J_B]$ are $10\sqrt\eps$-almost complete. 

If there is any edge $ib$ in $\blue{H}[J_A,J_B]$ with $i\in J_A$, then by moving $i$ out of $J_A$ and finding an arbitrary neighbour $a$ of $i$ in $J_A$, we get the structure required to apply Lemma~\ref{lemma:em3} (\textbf{EMd}) to find a blue copy of $\cT$ in $G$. Therefore, we can assume that $\blue{H}[J_A,J_B]$ is empty.

If $|A|\geq t_1+10\zeta n$, then after using Lemma~\ref{lemma:regularity:refine} to refine the clusters indexed by $J_A$ and $J_B$, we can embed $T$ in red using Lemma~\ref{lemma:hlt} (\textbf{H\L T}). If instead $|A|<t_1+10\zeta n$, then it follows from $|A|\geq(1-\zeta)(n+\tau_2)/2$ that $t_1\geq(1-50\zeta)(t_2+\tau_2)$. Also, we have $|B|\geq(1-\zeta)(n+\tau_2)-|A|\geq (1-25\zeta)(t_2+\tau_2)\geq(1-30\zeta)t_1$. Therefore, $J_A$ and $J_B$ show that $G$ contains a Type 3 $(100\zeta,T,\cT)$-extremal $(\eps,d)$-regular partition, so $G$ is Type 3 $(\mu,T,\cT)$-extremal by Lemma~\ref{lemma:regtoext}, and~\ref{main:4} holds. When~\ref{lemma:stage4:3} and~\ref{lemma:stage4:4} hold instead, a similar argument shows that $\tau_1\geq|A|-10\zeta n\geq(1-\zeta)(n+\tau_2)/2-10\zeta n\geq(1-20\zeta)(n+\tau_2)/2$, and $G$ is Type 4 $(\mu,T,\cT)$-extremal, so~\ref{main:5} holds.
\medskip

\noindent\textbf{Case II.} $|B|<t_2+100\zeta n$. Then $|A|\geq(1-\zeta)(n+\tau_2)-|B|>t_2+500\zeta n$.

If the maximum matching $M$ in $\red{H}[I_A]$ has size at least $101\zeta n/m$, then we can move one side of a submatching of $M$ with size $100\zeta n/m$ out of $I_A$ to obtain the structure needed to apply Lemma~\ref{lemma:em2a} (\textbf{EMa}) to find a red copy of $T$. Otherwise, let $I_A'=I_A\setminus V(M)$, let $A'=\cup_{i\in I_A'}V_i$, and observe that $|I_A'|\geq |I_A|-202\zeta n/m$. Then, $\red{H}[I_A']$ is empty and so $\blue{H}[I_A']$ contains at most $\eps k^2$ non-edges. 

Note that $|A'|=|I_A'|m\geq(1-400\zeta)(t_1+\tau_2)$. If $|A'|=|I_A'|m\geq(1+5\zeta)\nu$, then we can easily find the structure to apply Lemma~\ref{lemma:em2c} (\textbf{EMc}) to embed $\cT$ in $\blue{G}[A']$. Thus, we may assume that $|A'|\leq(1+5\zeta)\nu$, so $\nu\geq(1-400\zeta)(t_1+\tau_2)/(1+5\zeta)\geq(1-1000\zeta)(t_1+\tau_2)$, and $|B|\geq(1-\zeta)(n+\tau_2)-|A'|-202\zeta n\geq(1-1500\zeta)t_2$. Let $\zeta\ll\zeta'\ll d$. If at least $2\zeta' n/m$ vertices in $I_A'$ have at least $2\zeta' n/m$ blue neighbours in $I_B$, then we can greedily find the structure to apply Lemma~\ref{lemma:em2c} (\textbf{EMc}) to embed $\cT$ in $\blue{G}[A'\cup B]$. Otherwise, we can find $J_A\subset I_A'$ and $J_B\subset I_B$ such that $|J_A|\geq(1-3\zeta')(t_1+\tau_2)/m\geq(1-3\zeta')\nu/m$, $|J_B|\geq(1-50\sqrt{\zeta'})t_2/m$, and $\blue{H}[J_A,J_B]$ is $10\sqrt{\zeta'}$-almost empty. This shows that $G$ contains a Type 2 $(50\sqrt{\zeta'},T,\cT)$-extremal $(\eps,d)$-regular partition, so $G$ is Type 2 $(\mu,T,\cT)$-extremal by Lemma~\ref{lemma:regtoext}, and~\ref{main:3} holds. When~\ref{lemma:stage4:3} and~\ref{lemma:stage4:4} hold instead, a similar argument shows that $G$ is Type 1 $(\mu,T,\cT)$-extremal, so~\ref{main:2} holds.
\end{proof}

\subsection{Stage 3}
\begin{lemma}[\textbf{Stage 3}]\label{lem:stage3}
Let $1/n,1/\nu\ll 1/m\ll c\ll1/k\ll\eps\ll\zeta\ll d\ll\mu\ll1$. Let $T$ be an $n$-vertex tree with $\Delta(T)\leq cn$ and bipartition class sizes $t_1\geq t_2\geq(t_1-2)/3$. Let $\cT$ be a $\nu$-vertex tree with $\Delta(\cT)\leq c\nu$ and bipartition class sizes $\tau_1\geq\tau_2\geq(t_1-1)/2$. Suppose $\nu\geq t_1\geq\tau_1$, $\tau_2\geq t_2$, and $t_2+\tau_2\geq t_1-1$.

Let $G$ be a red/blue coloured complete graph that contains a coloured $\eps$-regular partition $V_1\cup\cdots\cup V_k$ with $|V_1|=\cdots=|V_k|=m$ and corresponding blue/red coloured $(\eps,d)$-reduced graph $H$. Suppose there is a partition $[k]=I_A\cup I_B\cup I_C$ such that either~\emph{\ref{lemma:stage3:1}} and~\emph{\ref{lemma:stage3:2}} hold or~\emph{\ref{lemma:stage3:3}} and~\emph{\ref{lemma:stage3:4}} hold  (see \textbf{\emph{C-situation}} in Figure~\ref{fig:stages}).
\stepcounter{propcounter}
\begin{enumerate}[label = \emph{\textbf{\Alph{propcounter}\arabic{enumi}}}]
    \item\labelinthm{lemma:stage3:1} $|I_A|m,|I_B|m\geq(1-\zeta)t_2$ and $|I_A|m+|I_B|m=(1-\zeta)(t_1+t_2)$.
    \item\labelinthm{lemma:stage3:2} $\red{H}[I_A,I_B]$ is $\zeta$-almost complete.
\end{enumerate}
\stepcounter{propcounter}
\begin{enumerate}[label = \emph{\textbf{\Alph{propcounter}\arabic{enumi}}}]
    \item\labelinthm{lemma:stage3:3} $|I_A|m,|I_B|m\geq(1-\zeta)\tau_2$ and $|I_A|m+|I_B|m=(1-\zeta)(t_1+\tau_2)$.
    \item\labelinthm{lemma:stage3:4} $\blue{H}[I_A,I_B]$ is $\zeta$-almost complete.
\end{enumerate}
Then, at least one of~\emph{\ref{main:1}}--\emph{\ref{main:5}} holds.
\end{lemma}
\begin{proof}
We will only prove the case when~\ref{lemma:stage3:1} and~\ref{lemma:stage3:2} hold. The case when~\ref{lemma:stage3:3} and~\ref{lemma:stage3:4} hold has essentially the same proof after exchanging the role of $t_2$ and $\tau_2$, and red and blue, so we will only comment on the one minor difference during the proof. 

Let $\zeta\ll\zeta'\ll\eta\ll\eta'\ll d$. To further avoid repetition, we first prove the following two claims dealing with two commonly occurring structures.

\begin{claim}\label{stage3claimA}
Suppose there exist disjoint $J_A,J_B,J_C\subset[k]$, such that the following hold. 
\stepcounter{propcounter}
\begin{enumerate}[label = \emph{\textbf{\Alph{propcounter}\arabic{enumi}}}]
    \item\labelinthm{lemma:stage3A:1} $|J_A|=(1-\zeta')t_1/m$, $|J_B|=(1-\zeta')t_2/m$, and $|J_C|=(1-\zeta')\tau_2/m$.
    \item\labelinthm{lemma:stage3A:2} $\red{H}[J_A,J_B]$ is $\zeta$-almost complete.
    \item\labelinthm{lemma:stage3A:3} $\red{H}[J_B,J_C]$ contains a matching with size $100\zeta' t_2/m$.
\end{enumerate}
Then, at least one of~\emph{\ref{main:1}}--\emph{\ref{main:5}} holds.
\end{claim}
\begin{proof}[Proof of Claim~\ref{stage3claimA}]
If at least $100\zeta' t_2/m$ vertices in $J_A$ have at least $200\zeta't_2/m$ red neighbours in $J_C$, then we can greedily find a matching of size $100\zeta' t_2/m$ in $\red{H}[J_A,J_C]$ disjoint from the matching given by~\ref{lemma:stage3A:3}. This allows us to apply Lemma~\ref{lemma:em2c} (\textbf{EMc}) to find a red copy of $T$ in $G$. Otherwise, at most $100\zeta' t_2/m$ vertices in $J_A$ have at least $200\zeta' t_2/m$ red neighbours in $J_C$, so we can find $J_A'\subset J_A$ and $J_C'\subset J_C$ with $|J_A'|\geq(1-200\zeta')t_1/m$ and $|J_C'|\geq(\tau_2-20\sqrt{\zeta'}t_2)/m\geq(1-20\sqrt{\zeta'})\tau_2/m$, such that $\blue{H}[J_A',J_C']$ is $20\sqrt{\zeta'}$-almost complete.

If at most $200\eta t_2/m$ vertices in $J_B$ have at least $200\eta t_2/m$ blue neighbours in $J_C'$, then we can find $J_B'\subset J_B$ and $J_C''\subset J_C'$, such that $|J_B'|\geq(1-300\eta)t_2/m$, $|J_C''|\geq(1-20\sqrt\eta)\tau_2/m$, and $\red{H}[J_A\cup J_C'',J_B']$ is $20\sqrt\eta$-almost complete. This gives the required structure (\textbf{D-situation}) in red to apply Lemma~\ref{lem:stage4} (\textbf{Stage 4}) to finish the proof.

Thus, we may assume that at least $200\eta t_2/m$ vertices in $J_B$ have at least $200\eta t_2/m$ blue neighbours in $J_C'$, so we can find a blue matching of size $100\eta t_2/m$ in $\blue{H}[J_B,J_C']$ disjoint from the red matching given by~\ref{lemma:stage3A:3}. Arbitrarily pick a matching of size $10(\zeta'+\eta)t_2/m$ in $H[J_A]$. By pigeonhole, it either contains a red matching $\red{M}$ of size $10\zeta' t_2/m$ or a blue matching $\blue{M}$ of size $10\eta t_2/m$. In the former case, moving vertices on one side of $\red{M}$ out of $J_A$, and using~\ref{lemma:stage3A:2},~\ref{lemma:stage3A:3}, we have the structure to apply Lemma~\ref{lemma:em2c} (\textbf{EMc}) to find a red copy of $T$ in $G$. In the latter case, moving vertices on one side of $\blue{M}$ out of $J_A$, and using that $\blue{H}[J_A,J_C']$ is $20\sqrt{\zeta'}$-almost complete, we can again apply Lemma~\ref{lemma:em2c} (\textbf{EMc}) to find a blue copy of $\cT$ in $G$. 
\renewcommand{\qedsymbol}{$\boxdot$}
\end{proof}
\renewcommand{\qedsymbol}{$\square$}

\begin{claim}\label{stage3claimB}
Suppose there exist disjoint $J_A,J_B,J_C\subset[k]$, such that the following hold. 
\stepcounter{propcounter}
\begin{enumerate}[label = \emph{\textbf{\Alph{propcounter}\arabic{enumi}}}]
    \item\labelinthm{lemma:stage3B:1} $|J_A|\geq(t_2+200\zeta'n)/m$, $|J_B|\geq(t_2-\zeta'n)/m$, $|J_A|+|J_B|=(1-\zeta')n/m$, and $|J_C|=(1-\zeta')\tau_2/m$.
    \item\labelinthm{lemma:stage3B:2} $\red{H}[J_A,J_B]$ is $\zeta$-almost complete.
    \item\labelinthm{lemma:stage3B:3} $\red{H}[J_A,J_C]$ contains a matching with size $100\zeta'n/m$.
\end{enumerate}
Then, at least one of~\emph{\ref{main:1}}--\emph{\ref{main:5}} holds.
\end{claim}
\begin{proof}[Proof of Claim~\ref{stage3claimB}]
If there is an edge in either $\red{H}[J_A]$ or $\red{H}[J_B]$, then we can apply Lemma~\ref{lemma:em2a} (\textbf{EMa}) or Lemma~\ref{lemma:em2b} (\textbf{EMb}), respectively, to find a red copy of $T$ in $G$. Thus, by removing at most $10\sqrt\eps$-fraction of vertices from $J_A$ and $J_B$, we may assume that $\blue{H}[J_A]$ and $\blue{H}[J_B]$ are both $10\sqrt\eps$-almost complete.

Suppose there exists a set of $10\eta t_2/m$ vertices in $J_C$, each of which has at least $10\eta t_2/m$ blue neighbours in both $J_A$ and $J_B$. Then, we can find a blue matching $M_1$ of size $\eta^2 t_2/m$ between $J_A$ and $J_C$ with a vertex in $V(M_1)\cap J_A$ adjacent to every vertex in $J_C':=V(M_1)\cap J_C$. Greedily, we can also find another blue matching between $J_C'$ and $J_B$ covering $J_C'$. This allows us to apply Lemma~\ref{lemma:em3} (\textbf{EMd}) to find a blue copy of $\cT$ in $G$. Here, we used $n\geq\nu$, so $|J_A|+|J_B|\geq(1-\zeta')\nu/m$. In the other case when~\ref{lemma:stage3:3} and~\ref{lemma:stage3:4} hold, we instead use $\tau_2\geq t_2$, so from the corresponding assumptions that $|J_A|+|J_B|\geq(1-\zeta')(t_1+\tau_2)/m$, we have $|J_A|+|J_B|\geq(1-\zeta')n/m$, and thus enough space to embed $T$ in red using Lemma~\ref{lemma:em3} (\textbf{EMd}).

Thus, we may now assume that there exist two disjoint subsets $J_A^+,J_B^+\subset J_C$ containing all but at most $10\eta t_2/m$ vertices in $J_C$, such that every vertex in $J_A^+$ has at most $10\eta t_2/m$ blue neighbours in $J_B$, and every vertex in $J_B^+$ has at most $10\eta t_2/m$ blue neighbours in $J_A$. Using~\ref{lemma:stage3B:2}, and by removing at most $5\sqrt\eta t_2/m$ vertices from each of $J_A$ and $J_B$ to obtain $J_A^-$ and $J_B^-$, respectively, we can ensure that both $\red{H}[J_A^-,J_B^-\cup J_B^+]$ and $\red{H}[J_B^-,J_A^-\cup J_A^+]$ are $5\sqrt\eta$-almost complete. Note that if $|J_A^+|\leq\eta't_2/m$, then $|J_A^-\cup J_B^-\cup J_B^+|\geq(1-2\eta')(n+\tau_2)/m$, so $J_A^-$ and $J_B^-\cup J_B^+$ form the required red structure (\textbf{D-situation}) to apply Lemma~\ref{lem:stage4} (\textbf{Stage 4}) to finish the proof. Thus, we can assume that $|J_A^+|\geq\eta't_2/m$, and similarly $|J_B^+|\geq\eta't_2/m$.

As in the beginning of this proof, we may further assume that $\blue{H}[J_A^-\cup J_A^+]$ and $\blue{H}[J_B^-\cup J_B^+]$ are both $10\sqrt\eps$-almost complete, as the presence of any red edge within either of them allows us to apply Lemma~\ref{lemma:em2a} (\textbf{EMa}) or Lemma~\ref{lemma:em2b} (\textbf{EMb}) to find a red copy of $T$ in $G$. If there is any blue edge between $J_A^-\cup J_A^+$ and $J_B^-\cup J_B^+$, then we can move one end of this blue edge out and then find a blue copy of $\cT$ using Lemma~\ref{lemma:em3} (\textbf{EMd}). Therefore, $\red{H}[J_A^-\cup J_A^+,J_B^-\cup J_B^+]$ contains the red structure (\textbf{D-situation}) needed to apply Lemma~\ref{lem:stage4} (\textbf{Stage 4}) to finish the proof.
\renewcommand{\qedsymbol}{$\boxdot$}
\end{proof}
\renewcommand{\qedsymbol}{$\square$}

Now we can carry out \textbf{Stage 3}. If at most $2\eta n/m$ vertices in $I_A\cup I_B$ have at least $2\eta \tau_2/m$ red neighbours in $I_C$, then we can find $I_{AB}\subset I_A\cup I_B$ and $I_C'\subset I_C$ with $|I_{AB}|\geq(1-3\eta)n/m$ and $|I_C'|\geq(1-5\sqrt\eta)\tau_2/m$, such that $\blue{G}[I_{AB}, I_C']$ is $5\sqrt\eta$-almost complete. Therefore, clusters indexed by $I_{AB}$ and $I_C'$ form the blue structure (\textbf{D-situation}) required to apply Lemma~\ref{lem:stage4} (\textbf{Stage 4}) to finish the proof. Thus, we may assume that at least $2\eta n/m$ vertices in $I_A\cup I_B$ have at least $2\eta\tau_2/m$ red neighbours in $I_C$, from which it follows that there is a red matching of size $\eta\tau_2/m\geq\eta t_2/m$ between $I_A$ and $I_C$ or between $I_B$ and $I_C$.

\medskip

\noindent\textbf{Case I.} $|I_A|,|I_B|\geq(t_2+200\zeta n)/m$. Without loss of generality, assume that there is a red matching of size $\eta t_2/m$ between $I_A$ and $I_C$. Then, we can apply Claim~\ref{stage3claimB} to finish the proof.

\medskip

\noindent\textbf{Case II.} Without loss of generality, assume that $|I_B|<(t_2+200\zeta n)/m$, so $|I_A|\geq(t_1-201\zeta n)/m\geq(1-500\zeta)t_1/m$ from~\ref{lemma:stage3:1}. If there is a red matching of size $\eta t_2/m$ between $I_B$ and $I_C$, then we can finish the proof using Claim~\ref{stage3claimA}. 

Thus, we may assume that there is a red matching of size $\eta t_2/m$ between $I_A$ and $I_C$. If $t_1\leq(1+2000\zeta)t_2$, then $|I_B|\geq(1-\zeta)t_2/m\geq(1-2001\zeta)t_1/m$, so we can apply Claim~\ref{stage3claimA} to finish the proof. If instead $t_1>(1+2000\zeta)t_2$, then $|I_A|\geq(1-500\zeta)t_1/m\geq(t_2+200\zeta n)/m$, so we can apply Claim~\ref{stage3claimB} to finish the proof.
\end{proof}

\subsection{Stage 2}
\begin{lemma}\label{lem:stage2}
Let $1/n,1/\nu\ll 1/m\ll c\ll1/k\ll\eps\ll\zeta\ll d\ll\mu\ll1$. Let $T$ be an $n$-vertex tree with $\Delta(T)\leq cn$ and bipartition class sizes $t_1\geq t_2\geq(t_1-2)/3$. Let $\cT$ be a $\nu$-vertex tree with $\Delta(\cT)\leq c\nu$ and bipartition class sizes $\tau_1\geq\tau_2\geq(t_1-1)/2$. Suppose $\nu\geq t_1\geq\tau_1$, $\tau_2\geq t_2$, and $t_2+\tau_2\geq t_1-1$. 

Let $G$ be a red/blue coloured complete graph that contains a coloured $\eps$-regular partition $V_0\cup V_1\cup\cdots\cup V_{k}$ with corresponding red/blue coloured $(\eps,d)$-reduced graph $H$, such that $|V_0|=\cdots=|V_{k}|=m$, and $km\geq t_1+t_2+\tau_2-2\zeta t_1$. Suppose there is a partition $[k]=I_A\cup I_B\cup I_C$ such that the following hold (see \textbf{\emph{B-situation}} in Figure~\ref{fig:stages}).
\stepcounter{propcounter}
\begin{enumerate}[label = \emph{\textbf{\Alph{propcounter}\arabic{enumi}}}]
    \item\labelinthm{lemma:stage2:1} $|I_A|m,|I_B|m\geq(1-\zeta)t_2$ and $|I_A|m+|I_B|m=t_2+\tau_2-\zeta t_1$.
    \item\labelinthm{lemma:stage2:2} In $\red{H}$, 0 is adjacent to every $i\in I_A\cup I_B$.
    \item\labelinthm{lemma:stage2:3} $\red{H}[I_A,I_B]$ is $\zeta$-almost complete.
\end{enumerate}
Then, at least one of~\emph{\ref{main:1}}--\emph{\ref{main:5}} holds.
\end{lemma}
\begin{proof}
Let $\eps\ll\gamma\ll\zeta$. Refine every cluster $V_i$ with $i\in I_A$ or $i\in I_B$ down to clusters of size $\gamma t_2m/t_1$, and label these new clusters as $\{V_i':i\in J_A\}$ and $\{V_i':i\in J_B\}$, respectively. Refine every cluster $V_i$ with $i\in I_C$ down to clusters of size $\gamma m$ and label the new clusters as $\{V_i':i\in J_C\}$. In this refinement process, at most $O(\gamma n)$ covered vertices are lost. Let $H'$ be the new reduced graph on $J_A\cup J_B\cup J_C$, where for each $\ast\in\{\text{red},\text{blue}\}$, $ij\in E(H_\ast')$ if and only if $G_\ast[V_i',V_j']$ is $(\sqrt\eps,d-\eps)$-regular. By Lemma~\ref{lemma:regularity:1}, $H'[J_A,J_B]$ is $\zeta$-almost complete. Let $M$ be a maximum red matching in $H'$ between $J_A\cup J_B$ and $J_C$.

\medskip

\noindent\textbf{Case I.} $|V(M)\cap J_C|\leq(t_1-\tau_2+10\zeta t_1)/\gamma m$. Let $X=(J_A\cup J_B)\cap V(M), X'=(J_A\cup J_B)\setminus V(M)$, $Y=J_C\cap V(M)$, and $Y'=J_C\setminus V(M)$. By Lemma~\ref{lemma:improvemaximummatching}, there are partitions $X=X^+\cup X^-\cup \overline{X}$ and $Y=Y^+\cup Y^-\cup \overline{Y}$ such that $M$ matches $X^+$ with $Y^-$, $X^-$ with $Y^+$, and $\overline{X}$ with $\overline{Y}$, and $\red{H}[X'\cup X^-,Y'\cup Y^-\cup\overline{Y}]$ is empty. From assumption, $|X|=|Y|\leq(t_1-\tau_2+10\zeta t_1)/\gamma m$, and the following hold.
\[|X'\cup X^-|\geq|X'|\geq(t_2+\tau_2-\zeta t_1)/\gamma m-(t_1-\tau_2+10\zeta t_1)/\gamma m\geq(\tau_2-12\zeta t_1)/\gamma m.\]
\[|Y'\cup Y^-\cup\overline{Y}|=|I_C|-|Y^+|\geq|I_C|-|Y|\geq(1-\zeta)t_1/\gamma m-(t_1-\tau_2+10\zeta t_1)/\gamma m\geq(\tau_2-12\zeta t_1)/\gamma m.\]
\begin{align*}
|X'\cup X^-\cup Y'\cup Y^-\cup\overline{Y}|&=|X'\cup Y^+\cup Y'\cup Y^-\cup\overline{Y}|=|(J_A\cup J_B)\setminus V(M)|+|J_C|\\
&\geq(t_1+t_2+\tau_2-2\zeta t_1)/\gamma m-(t_1-\tau_2+10\zeta t_1)/\gamma m\geq(t_1+\tau_2-13\zeta t_1)/\gamma m.
\end{align*}
Hence, the clusters indexed by $X'\cup X^-$ and $Y'\cup Y^-\cup\overline{Y}$ contain a blue \textbf{C-situation}, so we can use Lemma~\ref{lem:stage3} (\textbf{Stage 3}) to complete the proof.

\medskip

\noindent\textbf{Case II.} $|V(M)\cap J_C|\geq(t_1-\tau_2+10\zeta t_1)/\gamma m$. For brevity, let $m'=\gamma t_2m/t_1$ be the common size of clusters in $J_A\cup J_B$. Let $M'$ be a red matching between $J_A\cup J_B$ and $J_C$, such that $|V(M')\cap J_C|=|V(M')\cap J_A|+|V(M')\cap J_B|=(t_1-\tau_2+10\zeta t_1)/\gamma m$. It follows that 
\[|V(M')\cap J_A|m'+|V(M')\cap J_B|m'=(t_1-\tau_2+10\zeta t_1)t_2/t_1=t_2-\tau_2t_2/t_1+10\zeta t_2,\]
\begin{align*}
|J_A\setminus V(M)'|m'+|J_B\setminus V(M')|m'&=(t_2+\tau_2-\zeta t_1)-(t_2-\tau_2t_2/t_1+10\zeta t_2)\\
&=\tau_2+\tau_2t_2/t_1-\zeta t_1-10\zeta t_2\geq(1+t_1/t_2)(\tau_2t_2/t_1-6\zeta t_2).
\end{align*}

\textbf{Case II.1.} $|J_A\setminus V(M')|m',|J_B\setminus V(M')|m'\geq\tau_2t_2/t_1-6\zeta t_2$. 
Then, we can find disjoint subsets $J_{A,1},J_{A,2}\subset J_A\setminus V(M')$ and $J_{B,1},J_{B,2}\subset J_B\setminus V(M')$, such that $|J_{A,1}|/|J_{B,2}|=|J_{B,1}|/|J_{A,2}|\geq t_1/t_2$, each of $|J_{A,1}|,|J_{A,2}|,|J_{B,1}|,|J_{B,2}|$ is at least $\eps n/m'$, and their total size is at least $(1+t_1/t_2)(\tau_2t_2/t_1-8\zeta t_2)/m'$. Moreover, since $\blue{H'}[J_A,J_B]$ is $\zeta$-almost complete, by choosing the subsets above randomly and using Lemma~\ref{lemma:chernoff}, we can ensure that both $\blue{H'}[J_{A,1},J_{B,2}]$ and $\blue{H'}[J_{A,2},J_{B,1}]$ are $10\zeta$-almost complete.
Then, for some $\eps\ll\gamma'\ll\gamma$, use Lemma~\ref{lemma:regularity:refine} to further appropriately refine these clusters along with those in $M'$ into two sets of smaller clusters of sizes $\gamma't_2m/t_1$ and $\gamma'm$, respectively, which can be paired up into $(\sqrt[4]\eps,d-2\sqrt\eps)$-regular pairs. Together, these new clusters contain at least \[(1-\gamma')(1+t_2/t_1)((\tau_2-8\zeta t_1)+(t_1-\tau_2+10\zeta t_1))\geq(1+\zeta)n\]
vertices. Finally, note that by~\ref{lemma:stage2:2}, each new cluster of size $\gamma't_2m/t_1$ forms an $(\sqrt\eps,d-\eps)$-regular pair with $V_0$, so we have the structure to apply Lemma~\ref{lemma:hlt} (\textbf{H\L T}) to find a red copy of $T$.

\textbf{Case II.2.} One of $|J_A\setminus V(M')|m'$ and $|J_B\setminus V(M')|m'$ is at most $\tau_2t_2/t_1-6\zeta t_2$. Without loss of generality, say $|J_A\setminus V(M')|m'\leq\tau_2t_2/t_1-6\zeta t_2$. 
Moreover, pick $M'$ so that it maximises $|J_A\setminus V(M')|$ subject to $|J_A\setminus V(M')|m'\leq\tau_2t_2/t_1-6\zeta t_2$. Then, $|V(M')\cap J_A|m'\geq(1-\zeta)t_2-\tau_2t_2/t_1+6\zeta t_2\geq5\zeta t_2$, so in particular $|V(M')\cap J_A|\geq1$. Also, $|J_B\setminus V(M')|m'\geq\tau_2-6\zeta t_1$. Let $J_D$ be the vertices in $J_C$ matched by $M'$ to the vertices in $J_B$. 

Suppose there is a red edge in $H'$ between $J_B\setminus V(M')$ and $J_C\setminus J_D$. Let $M''$ be the matching obtained by replacing an edge in $M'$ adjacent to $J_A$ with this edge. Then, $|J_A\setminus V(M'')|=1+|J_A\setminus V(M')|$, so in order to avoid contradicting the maximality of $M'$, we must have $|J_A\setminus V(M'')|m'\geq\tau_2t_2/t_1-6\zeta t_2$. Note that $|J_B\setminus V(M'')|m'=|J_B\setminus V(M')|m'-m'\geq\tau_2-6\zeta t_1-m'\geq\tau_2t_2/t_1-6.1\zeta t_2$. Therefore, we can use the same proof as in \textbf{Case II.1.} with 6.1 in place of 6 to find a red copy of $T$.

Otherwise, there is no red edge in $H'$ between $J_B\setminus V(M')$ and $J_C\setminus J_D$. Thus, $\blue{H'}[J_B\setminus V(M'),J_C\setminus J_D]$ contains at most $\eps|H'|^2$ non-edges. Moreover, $|J_D|=|V(M')\cap J_B|\leq6\zeta t_1/m'=6\zeta t_1^2/\gamma t_2m\leq24\zeta t_1/\gamma m$, so $|J_C\setminus J_D|\geq(1-25\zeta)t_1/\gamma m$. Therefore, the clusters indexed by $J_B\setminus V(M')$ and $J_C\setminus J_D$ cover at least $\tau_2-6\zeta t_1$ and $(1-25\zeta)t_1$ vertices, respectively. After removing clusters with low degrees in $\blue{H'}[J_B\setminus V(M'),J_C\setminus J_D]$ and refining all remaining clusters down to a smaller common size, we get a \textbf{C-situation} in blue, and can finish the proof with Lemma~\ref{lem:stage3} (\textbf{Stage 3}).
\end{proof}

\subsection{Stage 1}\label{sec:stage1}
\begin{lemma}[\textbf{Stage 1}]\label{lem:stage1}
Let $1/n,1/\nu\ll 1/m\ll c\ll1/k\ll\eps\ll \zeta\ll d\ll\mu\ll1$. Let $T$ be an $n$-vertex tree with $\Delta(T)\leq cn$ and bipartition class sizes $t_1\geq t_2\geq(t_1-2)/3$. Let $\cT$ be a $\nu$-vertex tree with $\Delta(\cT)\leq c\nu$ and bipartition class sizes $\tau_1\geq\tau_2\geq(t_1-1)/2$. Suppose $\nu\geq t_1\geq\tau_1$, $\tau_2\geq t_2$, and $t_2+\tau_2\geq t_1-1$. Let $G$ be a red/blue coloured complete graph with at least $n+\tau_2$ vertices.
Then, at least one of~\emph{\ref{main:1}}--\emph{\ref{main:5}} holds.
\end{lemma}
\begin{proof}
Let $\alpha=(1+\eps/2)t_1/t_2$ and $\beta=(1+\eps/2)t_1/\tau_2$, so $\alpha\geq\beta>1$, and $1/\alpha+1/\beta\geq 1-\eps$. By Theorem~\ref{theorem:regularity}, $G$ contains a coloured $\eps$-regular partition $V_1\cup\cdots\cup V_{k}$ with $|V_1|=\cdots=|V_{k}|=m$ and $km=(1+1/\alpha+1/\beta-10\eps)t_1$. Let $H$ be the corresponding red/blue coloured $(\eps,d)$-reduced graph, and note that $H$ contains at most $\eps k^2$ non-edges. By removing at most $\sqrt{\eps}k\ll\zeta k$ clusters if necessary, we may assume that $|H|\geq(1+1/\alpha+1/\beta-2\zeta)t_1/m$, and every vertex in $H$ has at most $\zeta t_1/m$ non-neighbours in $H$. Let $V=V(H)$ and $k'=t_1/m$. 

Let $v\in V$ be arbitrary. Then, either $\red{N}(v,V)\geq(1/2+1/2\alpha+1/2\beta-\zeta)k'$, or $\blue{N}(v,V)\geq(1/2+1/2\alpha+1/2\beta-\zeta)k'$. First assume that $\red{N}(v,V)\geq(1/2+1/2\alpha+1/2\beta-\zeta)k'$. Let $Q=\red{N}(v,V)$, $R=V\setminus(Q\cup\{v\})$, and let $A,B,C$ be given by applying Lemma~\ref{lemma:ABC} to $\red{\vv{H}}(v)$. If $w_{\text{red},\max}(\red{\vv{H}}(v))\geq(1+\zeta)k'$, then $G$ contains a red copy of $T$ by Lemma~\ref{lem:getHLT} and Lemma~\ref{lemma:hlt} (\textbf{H\L T}), so assume otherwise. Then, by~\ref{QR:1} and~\ref{QR:3}, $|Q\cap A|+|Q\cap C|+|R\cap A|>(1/\alpha+1/\beta-4\zeta)k'$.

\medskip

\noindent\textbf{Case I.} $Q\cap A=\varnothing$. Then, $|Q\cap B|+|Q\cap C|\geq(1/2+1/2\alpha+1/2\beta-\zeta)k'$. Since $w_{\text{red},\max}(\red{\vv{H}}(v))<(1+\zeta)k'$, $(1+\alpha)|Q\cap B|+|Q\cap C|\leq(1+1/\alpha)(1+\zeta)k'<(1+1/\alpha+2\zeta)k'$ by~\ref{ABC:6}. Thus, $|Q\cap B|\leq(1/2\alpha+1/2\alpha^2-1/2\alpha\beta+3\zeta)k'$, and so $|Q\cap C|\geq(1/2+1/2\beta-1/2\alpha^2+1/2\alpha\beta-4\zeta)k'$. Since $\alpha\geq\beta$, 
\[|Q\cap C|\geq(1/2\beta+1/2\beta+(1/\beta-1/\alpha)/2\alpha-4\zeta)k'\geq(1/\beta-4\zeta)k'.\] 
Next, by~\ref{ABC:8}, $|R\cap A|=|A|>|V|-1-(1+1/\alpha)(1+\zeta)k'\geq(1/\beta-5\zeta)k'$. Moreover, by~\ref{ABC:6}, $|R\cap C|+|B|=|C|-|Q\cap C|+|B|<(1+1/\alpha)(1+\zeta)k'-\alpha|B|-|Q\cap C|\leq(1+1/\alpha+2\zeta)k'-|Q|$, and so 
\[|Q\cap C|+|R\cap A|=|V|-|R\cap C|+|B|\geq(1/\beta-4\zeta)k'+(1/2+1/2\alpha+1/2\beta-\zeta)k'\geq(1+1/\beta-6\zeta)k'.\] Therefore, using~\ref{QR:2}, the clusters in $Q\cap C$ and $R\cap A$ together contain a \textbf{C-situation} in blue, so we can apply Lemma~\ref{lem:stage3} (\textbf{Stage 3}) to finish the proof.

\medskip

\noindent\textbf{Case II.} $|Q\cap A|\geq1$. Then, by~\ref{QR:2} and using that any $z\in Q\cap A$ is a blue neighbour of all but at most $\zeta k'$ vertices in $(Q\cap A)\cup(Q\cap C)\cup(R\cap A)$, we get that $|\blue{N}(z,V)|\geq(1/\alpha+1/\beta-5\zeta)k'$. 

Let $Q'=\blue{N}(z,V)$, $R'=V\setminus(Q'\cup\{z\})$, and let $A',B',C'$ be given by applying Lemma~\ref{lemma:ABC} to $\blue{\vv{H}}(z)$. If $w_{\text{blue},\max}(\blue{\vv{H}}(z))\geq (1+\zeta)k'$, then we can find a blue copy of $\cT$ by Lemma~\ref{lem:getHLT} and Lemma~\ref{lemma:hlt} (\textbf{H\L T}), so assume otherwise. Then, by~\ref{QR:1} and~\ref{QR:3}, $|Q'\cap A'|+|Q'\cap C'|+|R'\cap A'|>(1/\alpha+1/\beta-4\zeta)k'$.

\textbf{Case II.1.} $Q'\cap A'=\varnothing$. Then, $|Q'\cap B'|+|Q'\cap C'|\geq(1/\alpha+1/\beta-5\zeta)k'$. We proceed similarly to \textbf{Case I} above. Since $w_{\text{blue},\max}(\blue{\vv{H}}(z))<(1+\zeta)k'$, by~\ref{ABC:6}, $(1+\beta)|Q'\cap B'|+|Q'\cap C'|\leq(1+1/\beta)(1+\zeta)k'<(1+1/\beta+2\zeta)k'$. Thus, $|Q'\cap B'|\leq(1/\beta-1/\alpha\beta+7\zeta)k'$, and so $|Q'\cap C'|\geq(1/\alpha+1/\alpha\beta-12\zeta)k'\geq(1/\alpha-5\zeta)k'$. By~\ref{ABC:8}, $|R'\cap A'|=|A'|>|V|-1-(1+1/\beta)(1+\zeta)k'\geq(1/\alpha-5\zeta)k'$. Then, by~\ref{ABC:6}, $|R'\cap C'|+|B'|=|C'|-|Q'\cap C'|+|B'|<(1+1/\beta)(1+\zeta)k'-\beta|B'|-|Q'\cap C'|\leq(1+1/\beta+2\zeta)k'-|Q'|$. Thus, $|Q'\cap C'|+|R'\cap A'|\geq(1/\alpha-4\zeta)k'+|Q'|\geq(2/\alpha+1/\beta-9\zeta)k'\geq(1+1/\alpha-10\zeta)k'$. Therefore, the clusters in $Q'\cap C'$ and $R'\cap A'$ contain a \textbf{C-situation} in red, so we can use Lemma~\ref{lem:stage3} (\textbf{Stage 3}) to finish the proof.

\textbf{Case II.2.} $|Q'\cap A'|\geq1$. Since $\alpha\geq\beta$, by Lemma~\ref{lem:getBsituation}, we get a \textbf{B-situtaion} in red, and can finish the proof using Lemma~\ref{lem:stage2} (\textbf{Stage 2}). 

\medskip

The case when $\blue{N}(v,V)\geq(1/2+1/2\alpha+1/2\beta-\zeta)k'$ is similar, and the same proof mostly works after exchanging the roles of $\alpha$ and $\beta$, and red and blue, so we only comment on two instances where $\alpha\geq\beta$ is used in the proof above. 

Let $Q=\blue{N}(v,V)$, $R=V\setminus(Q\cup\{v\})$, and let $A,B,C$ be given by applying Lemma~\ref{lemma:ABC} to $\blue{\vv{H}}(v)$. If $Q\cap A=\varnothing$, then as in \textbf{Case I} above, we can similarly obtain $|Q\cap C|\geq(1/2+1/2\alpha-1/2\beta^2+1/2\alpha\beta-4\zeta)k'$. Then, we prove $|Q\cap C|\geq(1/\alpha-4\zeta)k'$ in a different way as follows.
\begin{align*}
|Q\cap C|\geq(1/2+1/2\alpha-1/2\beta^2+1/2\alpha\beta-4\zeta)k'&\geq(1/2\beta+1/2\alpha-1/2\beta^2+1/2\alpha\beta-4\zeta)k'\\
&=(1/\alpha-4\zeta+(1/2\beta-1/2\alpha)(1-1/\beta))k'\\
&\geq(1/\alpha-4\zeta)k',
\end{align*}
where the last inequality follows from $\alpha\geq\beta$. The rest of the proof in \textbf{Case I} does not use $\alpha\geq\beta$, so carries through directly. 

If $|Q\cap A|\geq1$, then we proceed as in \textbf{Case II}. The only place where $\alpha\geq\beta$ is used is \textbf{Case II.2}. In this setting, $Q'=\red{N}(z,V)$, $R'=V\setminus(Q'\cup\{z\})$, $A',B',C'$ are given by applying Lemma~\ref{lemma:ABC} to $\red{\vv{H}}(z)$, and $|Q'\cap A'|\geq1$. Then, any $y\in Q'\cap A'$ is a blue neighbour of all but at most $\zeta k'$ vertices in $(Q'\cap A')\cup(Q'\cap C')\cup(R'\cap A')$. In particular, $|\blue{N}(y,V)|\geq(1/\alpha+1/\beta-5\zeta)k'$. Let $Q''=\blue{N}(y,V)$, $R''=V\setminus(Q''\cup\{y\})$, and let $A'',B'',C''$ be given by applying Lemma~\ref{lemma:ABC} to $\blue{\vv{H}}(y)$. Then, we are in the same situation as \textbf{Case II} above with $y$ in place of $z$, and can finish the proof in the same way. 
\end{proof}

\section{Extremal part}\label{sec:extremal}
In this section, we show that even if $G$ is $(\mu,T,\cT)$-extremal, we can still find either a red copy of $T$ or a blue copy of $\cT$ in $G$. The cases when $G$ is Type 1, Type 2, Type 3, or Type 4 $(\mu,T,\cT)$-extremal will be proved in order from Section~\ref{sec:type1} to Section~\ref{sec:type4}. Combining these with Theorem~\ref{thm:stability2}, we complete the proof of Theorem~\ref{thm:main2}, and thus of Theorem~\ref{thm:main}. Before that, we collect and prove several preparatory tree embedding results, which allow us to embed trees with many bare paths or many leaves into almost complete graphs or almost complete bipartite graphs. Many proofs of this section are adapted from the extremal part of the proof of Theorem~\ref{thm:oldmain} in~\cite{MPY}. 

The first such embedding result is adapted from {\cite[Lemma 6.1]{MPY}}, and shows that a spanning tree with many bare paths can be embedded into an almost complete graph. 
\begin{lemma}\label{lemma:bare-paths}
Let $1/n\ll\mu\ll\xi\leq 1/100$. Let $H$ be a graph with at least $n$ vertices satisfying $\delta(H)\ge\xi n$ and $d(v)\geq|H|-\mu n$ for all but a set $W$ of at most $\mu n$ vertices $v\in V(H)$. Let $T$ be an $n$-vertex tree with $\Delta(T)\leq\mu n$ that contains $10\mu n$ vertex-disjoint bare paths of length 4. Then, for any $t\in V(T)$ not on any of these bare paths and any $u\in V(H)$, there is a copy of $T$ in $H$ with $t$ copied to $u$.
\end{lemma}
\begin{proof}
Let $\mathcal{P}$ be a set of $\ell=10\mu n$ vertex-disjoint bare paths of length 4 in $T$, and let $T'$ be the forest obtained by deleting the 3 internal vertices of every path in $\mathcal{P}$ from $T$, so in particular $|T'|=(1-30\mu)n$. 

Let $m=|W|$ if $u\not\in W$, and let $m=|W|-1$ otherwise, noting that $m\leq\mu n$. Let the vertices in $W$ that are not $u$ be $w_1,\ldots,w_m$. Since $\delta(H)\geq\xi n\gg m$, we can greedily find distinct vertices $x_1,y_1,\ldots,x_m,y_m\in V(H)\setminus W$, such that $x_iw_iy_i$ is a path of length 2 for every $i\in[m]$. Let $W^+=\cup_{i=1}^m\{x_i,y_i,w_i\}$, and let $H'=H-W^+$. 

Note that $d_{H'}(v)\geq|H'|-\mu n\geq|T'|$ for every $v\in V(H')$ except possibly $u$, and $d_{H'}(u)\geq\xi n-3m\geq\Delta(T)$. Thus, we can greedily find a copy $\psi(T')$ of $T'$ in $H'$, with $t$ copied to $u$. Let $u_1,v_1,\ldots,u_\ell,v_\ell$ be the copies of the endpoints of the $\ell$ paths in $\mathcal{P}$. To complete a copy of $T$, it suffices to find $u_iv_i$-paths of length 4 using new distinct internal vertices for each $i\in[\ell]$. 

Let $H''=H'-V(\psi(T'))=H-W^+-V(\psi(T'))$, and note that $\delta(H'')\geq|H''|-\mu n\geq|H''|/2$ because $|H''|\geq n-3m-(1-30\mu)n\geq27\mu n$. By Dirac's Theorem, there is a Hamilton cycle in $H''$. In particular, we can find distinct $x_{m+1},y_{m+1},w_{m+1},\ldots,x_{\ell},y_{\ell},w_{\ell}\in V(H'')$, such that $x_iw_iy_i$ is a path of length 2 for every $m+1\leq i\leq\ell$. 

Let $K$ be the auxiliary bipartite graph with bipartition classes $A=\{a_1,\ldots,a_\ell\}$ and $B=\{b_1,\ldots,b_\ell\}$, such that for any $i,j\in[\ell]$, $a_ib_j\in E(K)$ if and only if both $u_ix_j$ and $v_iy_j$ are edges in $H$. Since $\cup_{i=1}^\ell\{u_i,v_i,x_i,y_i\}\subset V(H)\setminus W$, for every $i\in[\ell]$, $d_K(a_i),d_K(b_i)\geq\ell-2\mu n$. Then, for any $I\subset A$ with $0<|I|\leq\ell-2\mu n$, we have $|N_K(I,B)|\geq\ell-2\mu n\geq |I|$, while for any $I\subset A$ with $|I|>\ell-2\mu n$, we have $|N_K(I,B)|=|B|\geq|I|$, as any $b\in B\setminus N_K(I,B)$ would satisfy $d_K(b)<2\mu n<\ell-2\mu n$, a contradiction. Thus, by Lemma~\ref{lemma:Hall}, there is a perfect matching in $K$, say matching $a_i$ with $b_{\sigma(i)}$ for every $i\in[\ell]$. This implies that $\psi(T')$ along with the paths $u_ix_{\sigma(i)}w_{\sigma(i)}y_{\sigma(i)}v_i$ for all $i\in[\ell]$ form a copy of $T$ in $H$, as required.
\end{proof}

The next result is a bipartite version of Lemma~\ref{lemma:bare-paths}. 
\begin{lemma}[{\cite[Lemma 6.2]{MPY}}]\label{lemma:bp:bare-paths}
Let $1/n\ll \mu\ll\xi\leq1/100$. Let $H$ be a bipartite graph with bipartition classes $U_1$ and $U_2$ such that $d(v,U_2)\ge|U_2|-\mu n$ for every $v\in U_1$, $d(v,U_1)\geq\xi n$ for every $v\in U_2$, and $d(v,U_1)\geq|U_1|-\mu n$ for all but a set $W$ of at most $\mu n$ vertices in $U_2$.

Let $T$ be an $n$-vertex forest with bipartition classes $V_1$ and $V_2$ such that $|V_j|\le |U_j|$ for each $j\in [2]$. Moreover, suppose that $T$ contains a set $\mathcal{P}$ of $10\mu n$ vertex-disjoint bare paths of length 4 whose endpoints are all in $V_2$. Then, for any $j\in [2]$, any $t\in V_j$ and $u\in U_j$, $H$ contains a copy of $T$ with $t$ copied to $u$, $V_i$ copied to $U_i$ for each $i\in [2]$, and $W$ covered by the central vertices of the bare paths in $\mathcal{P}$.
\end{lemma}

The next result shows that if a tree has many leaves and a sublinear maximum degree, then it can be embedded into an almost complete graph of the same size. This is similar to~{\cite[Lemma 3.2]{Y}} and is a less complicated version of several related results in~\cite{MPY}.

\begin{lemma}\label{lemma:leaves}
For every $\xi\leq1/100$, there exists $c=c_\xi>0$ such that the following holds. Let $1/n\ll\mu\ll\xi$. Let $H$ be a graph with at least $n$ vertices, one of which is $u$, such that $d(v)\geq|H|-\mu n$ for all $v\in V(H)\setminus\{u\}$, and $d(u)\geq\xi n$. Let $T$ be an $n$-vertex tree with $\Delta(T)\leq cn/\log n$, and suppose that $T$ contains at least $\xi n/10^3$ leaves. Then, for any $t\in V(T)$, there is a copy of $T$ in $H$ with $t$ copied to $u$. 
\end{lemma}
\begin{proof}
Let $L$ be a set of $\xi n/10^4$ leaves in $T$ not including $t$ or any neighbours of $t$. Let $P=\{p_1,\ldots,p_\ell\}$ be the set of parents of leaves in $L$, and let $d_i=d(p_i,L)$ for every $i\in[\ell]$. Order the vertices of $T-L$ using a breadth first search as $t=t_1,t_2,\ldots,t_m$, so that for every $2\leq i\leq m$, $t_i$ has a unique neighbour to its left in this ordering, and all neighbours of $t_1$ appear immediately after $t_1$. By relabelling if necessary, we may assume that $p_1,\ldots,p_\ell$ appear in this order.

Consider the following random embedding $\psi$ of $T-L$ into $H$. Let $\psi(t)=\psi(t_1)=u$. Then, for each $2\leq i\leq m$, let $j_i\in[i-1]$ be the unique index such that $t_{j_i}$ is a neighbour of $t_i$, and pick $\psi(t_i)$ uniformly at random from the set $Y_i=N(\psi(t_{j_i}))\setminus\{\psi(t_1),\ldots,\psi(t_{i-1})\}$. We claim that this procedure always succeeds in producing a random embedding of $T-L$. Indeed, if $j_i\geq2$, then $\psi(t_{j_i})\not=u$, so $|Y_i|\geq|L|-\mu n\geq\xi n/10^4-\mu n\geq2\mu n>0$. If $j_i=1$, then $|Y_i|\geq\xi n-\Delta(T)\geq2\mu n>0$

Now, note that for every $v\in V(H)\setminus\{u\}$ and every $i\in[\ell]$, if $p_i=t_{i'}$ and $Y_{i'}$ is defined as above, then
\[\mathbb{P}(v\not\in N(\psi(p_i))\mid\psi(t_1),\ldots,\psi(t_{i'-1}))\leq\frac{|Y_{i'}\setminus N(v)|}{|Y_{i'}|}\leq\frac{\mu n}{2\mu n}<\frac12.\]
Therefore, by Lemma~\ref{lemma:azuma},
\begin{align*}
\mathbb{P}\left(\sum_{i:v\in N(\psi(p_i))}d_i<|L|/4\right)&\leq\exp\left(-\frac{(|L|/4)^2}{2\frac{|L|}{cn/\log n}(cn/\log n)^2}\right)\\
&=\exp\left(-\frac{|L|\log n}{32cn}\right)\leq\exp\left(-\frac{\xi\log n}{10^6c}\right)=o(n^{-1}),
\end{align*}
provided that $c<\xi/10^6$. Therefore, by a union bound, with high probability, every $v\in V(H)\setminus\{u\}$ satisfies $\sum_{i:v\in N(\psi(p_i))}d_i\geq|L|/4$. Fix an embedding $\psi$ with this property.

Let $W\subset V(H)\setminus\{u\}$ be the set of remaining vertices, it suffices now to embed $L$ appropriately into $W$, noting that $|W|\geq|L|$. To do this, we verify Hall's condition. For any $\varnothing\not=I\subset[\ell]$, if $\sum_{i\in I}d_i\leq|L|-\mu n$, then $|N(\psi(\{p_i:i\in I\}),W)|\geq|W|-\mu n\geq\sum_{i\in I}d_i$ as $d(v)\geq|H|-\mu n$ for every $v\in V(H)\setminus\{u\}$. If $I\subset[\ell]$ satisfies $\sum_{i\in I}d_i>|L|-\mu n$, then $N(\psi(\{p_i:i\in I\}),W)=W$ and $|N(\psi(\{p_i:i\in I\}),W)|=|W|\geq|L|\geq\sum_{i\in I}d_i$, because any $w\in W\setminus N(\psi(\{p_i:i\in I\}),W)$ would satisfy $\sum_{i:w\in N(\psi(p_i))}d_i\leq \mu n<|L|/4$, a contradiction. Hence, Hall's condition holds and we can use Lemma~\ref{lemma:hallmatching} to finish the embedding. 
\end{proof}

We also prove a bipartite version of Lemma~\ref{lemma:leaves}, which is again similar to some results in~\cite{MPY}. 
\begin{lemma}\label{lemma:bipartite:leaves}
For every $\xi\leq1/100$, there exists $c=c_\xi>0$ such that the following holds. 

Let $1/n\ll\mu\ll\xi$. Let $H$ be a bipartite graph with bipartition classes $U_1$ and $U_2$ such that $d(v,U_2)\ge|U_2|-\mu n$ for every $v\in U_1$, $d(v,U_1)\geq\xi|U_1|$ for every $v\in U_2$, and $d(v,U_1)\geq|U_1|-\mu n$ for all but a set $W_2$ of at most $\mu n$ vertices in $U_2$. 

Let $T$ be an $n$-vertex tree with $\Delta(T)\leq cn/\log n$ and bipartition classes $V_1$ and $V_2$, such that $|V_1|+n/10\leq|U_1|$ and $|V_2|\leq|U_2|$. If $T$ contains a set $L$ of $\xi n/10^5$ leaves in $V_2$, then there is an embedding of $T$ in $H$ with $V_i$ embedded into $U_i$ for each $i\in[2]$, and only vertices in $L$ embedded into $W_2$.
\end{lemma}
\begin{proof}
If $|U_2|\geq|V_2|+2\mu n$, then we can simply embed $T$ greedily, so assume otherwise. Let $P=\{p_1,\ldots,p_\ell\}$ be the set of parents of leaves in $L$, and let $d_i=d(p_i,L)$ for every $i\in[\ell]$. Note that $|P|\leq|L|=\xi n/10^5$. Order the vertices of $T-L$ as $t_1,t_2,\ldots,t_m$, such that $t_1\not\in P\cup L$, and for every $2\leq i\leq m$, $t_i$ has a unique neighbour to its left in this ordering. By relabelling if necessary, we may assume that $p_1,\ldots,p_\ell$ appear in this order.

Let $W_1\subset U_1$ be a uniformly random set of size $n/20$. Then, by Lemma~\ref{lemma:chernoff}, with high probability every $v\in U_2$ satisfies $d(v,W_1)\geq\xi n/50$.

Consider the following random embedding $\psi$ of $T-L$ into $H[U_1,U_2\setminus W_2]$. Pick $\psi(t_1)\in U_1\setminus W_1$ uniformly at random if $t_1\in V_1$, and pick $\psi(t_1)\in U_2\setminus W_2$ uniformly at random if $t_1\in V_2$. Then, for each $2\leq i\leq m$, let $j_i\in[i-1]$ be the unique index such that $t_{j_i}$ is a neighbour of $t_i$. If $t_i\not\in P$, pick $\psi(t_i)$ uniformly at random from the set $N(\psi(t_{j_i}))\setminus(W_1\cup W_2\cup\{\psi(t_1),\ldots,\psi(t_{i-1})\})$, which has size at least $|U_1|-|V_1|-|W_1|-\mu n>0$ if $t_i\in V_1$, and at least $|L|-|W_2|-\mu n>0$ if $t_i\in V_2$. If $t_i\in P$, then $t_i\in V_1$, and pick $\psi(t_i)$ uniformly at random from the set $Y_i=N(\psi(t_{j_i}),W_1)\setminus\{\psi(t_1),\ldots,\psi(t_{i-1})\}$, which has size at least $|W_1|-|P|-\mu n>0$. In particular, this procedure always succeeds in producing a random embedding of $T-L$.

Now, note that for every $v\in U_2$ and every $i\in[\ell]$, if $p_i=t_{i'}$ and $Y_{i'}$ is defined as above, then, using $d(v,W_1)\geq\xi n/50$ and that at most $|P|\leq\xi n/10^3$ vertices have been embedded into $W_1$, we have
\[\mathbb{P}(v\in N(\psi(p_i))\mid\psi(t_1),\ldots,\psi(t_{i'-1}))=\frac{|Y_{i'}\cap N(w)|}{|Y_{i'}|}\geq\frac{\xi n/50-\xi n/10^5-\mu n}{n/20}\geq\frac\xi{5}.\]
Therefore, by Lemma~\ref{lemma:azuma},
\begin{align*}
\mathbb{P}\left(\sum_{i:w\in N(\psi(p_i))}d_i<\xi|L|/10\right)&\leq\exp\left(-\frac{(\xi|L|/10)^2}{2\frac{|L|}{cn/\log n}(cn/\log n)^2}\right)\\
&=\exp\left(-\frac{\xi^2|L|\log n}{200cn}\right)\leq\exp\left(-\frac{\xi^3\log n}{10^8c}\right)=o(n^{-1}),
\end{align*}
provided that $c<\xi^3/10^8$. By a union bound, with high probability, $\sum_{i:w\in N(\psi(p_i))}d_i\geq\xi|L|/10$ for every $w\in V(H)$. Fix an embedding $\psi$ with this property.

Let $W\supset W_2$ be the set of remaining vertices in $U_2$, it suffices now to embed $L$ appropriately into $W$, noting that $|W|\geq|L|$. To do this, we verify Hall's condition. For any $\varnothing\not=I\subset[\ell]$, if $\sum_{i\in I}d_i\leq|L|-\mu n$, then $|N(\psi(\{p_i:i\in I\}),W)|\geq|W|-\mu n\geq\sum_{i\in I}d_i$ as $d(v,U_2)\geq|U_2|-\mu n$ for every $v\in U_1$. If $I\subset[\ell]$ satisfies $\sum_{i\in I}d_i>|L|-\mu n$, then $N(\psi(\{p_i:i\in I\}),W)=W$ and $|N(\psi(\{p_i:i\in I\}),W)|=|W|\geq|L|\geq\sum_{i\in I}d_i$, because any $w\in W\setminus N(\psi(\{p_i:i\in I\}),W)$ would satisfy $\sum_{i:w\in N(\psi(p_i))}d_i\leq \mu n<\xi|L|/10$, a contradiction. Hence, Hall's condition holds and we can use Lemma~\ref{lemma:hallmatching} to finish the embedding. Note that only vertices in $L$ are embedded in to $W_2$ in this embedding. 
\end{proof}

Finally, by combining Lemma~\ref{lemma:paths-leaf}, Lemma~\ref{lemma:bp:bare-paths} and Lemma~\ref{lemma:bipartite:leaves}, the following lemma embeds trees into almost complete bipartite graphs. 
\begin{lemma}\label{lemma:bipartite:anyleaves}
For every $\xi\leq1/100$, there exists $c=c_\xi>0$ such that the following holds. 

Let $1/n\ll\mu\ll\xi$. Let $H$ be a bipartite graph with bipartition classes $U_1$ and $U_2$ such that $d(v,U_2)\ge|U_2|-\mu n$ for every $v\in U_1$, $d(v,U_1)\geq\xi|U_1|$ for every $v\in U_2$, and $d(v,U_1)\geq|U_1|-\mu n$ for all but a set $W_2$ of at most $\mu n$ vertices in $U_2$. 

Let $T$ be an $n$-vertex tree with $\Delta(T)\leq cn/\log n$ and bipartition classes $V_1$ and $V_2$, such that $n/4\leq|V_1|+n/5\leq|U_1|$ and $n/100\leq|V_2|\leq|U_2|$. Then, there is an embedding of $T$ in $H$ with $V_i$ embedded into $U_i$ for each $i\in[2]$.
\end{lemma}
\begin{proof}
Let $L_1$ be the set of leaves of $T$ in $V_1$, and let $T'=T-L_1$. Note that $|T'|\geq|V_2|\geq n/100$.

Let $Z\subset U_1$ be a uniformly random subset of size $n/10$, and note that $|U_1|-|Z|\geq|V_1|+n/10$. Then, by Lemma~\ref{lemma:chernoff}, with high probability every $v\in U_2$ satisfies $d(v,Z)\geq\xi n/20$ and $d(v,U_1\setminus Z)\geq\xi|U_1\setminus Z|/2\geq\xi n/20$. Fix a realisation of $Z$ with this property.

If $T'$ contains at least $n/10^3$ vertex-disjoint bare paths of length 5, then it also contains at least $n/10^3$ vertex-disjoint bare paths of length 4 with all of their endpoints in $V_2$. By averaging, there is a collection of $100\mu n$ such bare paths, such that the set $X$ of their central vertices satisfies $N(X,L_1)\leq10^5\mu n^2/n\leq\xi n/20$.
Then, we can find a copy of $T'$ in $H[U_1\setminus Z,U_2]$ by Lemma~\ref{lemma:bp:bare-paths} with only vertices in $X$ embedded into $W_2$. Formally, this is done by letting $T''$ be the tree obtained by attaching $|L_1|$ new leaves to the end points of the $100\mu n$ chosen bare paths in $T'$, then apply Lemma~\ref{lemma:bp:bare-paths} to $T''$. To finish, note that $d(v,Z)\geq\xi n/20$ for every $v\in U_2$, and $d(v,U_1\setminus Z)\geq|U_1\setminus Z|-\mu n\geq|V_1|$ for every $v\in U_2\setminus W_2$, so we can greedily embed the leaves in $L_1$ adjacent to vertices in $X$ into $Z$, and the leaves in $L_1$ not adjacent to $X$ into $U_1\setminus Z$.

Otherwise, $T'$ contains at least $n/100$ leaves by Lemma~\ref{lemma:paths-leaf}, which must all belong to $V_2$. Again by averaging, we can find a set $X$ of $\xi n/10^4$ such leaves, such that $|N(X,L_1)|\leq\xi n^2/10^2n\leq\xi n/20$. 
By Lemma~\ref{lemma:bipartite:leaves}, formally applied to the tree $T''$ obtained by attaching $|L_1|$ leaves to vertices in $X$ while preserving $\Delta(T'')\leq cn/\log n$, we can find a copy of $T'$ in $H[U_1\setminus Z,U_2]$, with only vertices in $X$ embedded into $W_2$. Then, like above, we can finish the embedding of $T$ by greedily embedding leaves in $L_1$ into $U_1\setminus Z$ or $Z$, depending on whether they are adjacent to $X$.
\end{proof}

\subsection{Type 1}\label{sec:type1}
\begin{lemma}[\textbf{Type 1}]\label{lemma:type1}
There exists a constant $c>0$ such that the following holds. Let $1/n\leq1/\nu\ll\mu\leq1$. Let $T$ be an $n$-vertex tree with $\Delta(T)\leq cn/\log n$ and bipartition class sizes $t_1\geq t_2\geq(t_1-2)/3$. Let $\cT$ be a $\nu$-vertex tree with $\Delta(\cT)\leq c\nu/\log\nu$ and bipartition class sizes $\tau_1\geq\tau_2\geq(t_1-2)/2$. Suppose also that $\nu\geq t_1\geq\tau_1$, $\tau_2\geq t_2$ and $\tau_2+t_2\geq t_1-1$.

Let $G$ be a red/blue coloured complete graph with $\max\{n+\tau_2,2t_1\}-1$ vertices that is Type 1 $(\mu,T,\cT)$-extremal. That is, $G$ contains disjoint $U_1,U_2\subset V(G)$ such that $|U_1|\geq(1-\mu)n$, $|U_2|\geq(1-\mu)\tau_2$, $\blue{d}(v,U_1)\leq\mu n$ for every $v\in U_1$, and $\red{d}(v,U_{3-i})\leq\mu n$ for every $i\in[2]$ and $v\in U_i$. Then, $G$ contains either a red copy of $T$ or a blue copy of $\cT$.
\end{lemma}
\begin{proof}
Let $\xi=1/100$, and let $c$ be smaller than the constants given by applying Lemma~\ref{lemma:leaves} and Lemma~\ref{lemma:bipartite:anyleaves} with $\xi$. Let $U_1^+=\{v\in V(G):\red{d}(v,U_1)\geq|U_1|-\xi n\}$ and $U_2^+=V(G)\setminus U_1^+$. Then, $\blue{d}(v,U_1)\geq\xi n$ for every $v\in U_2^+$. Note also that $U_1\subset U_1^+$ and $U_2\subset U_2^+$. Let $k=|U_1^+|-n$, so $|k|\leq2\mu n$ and $|U_2^+|\geq\tau_2-k-1$.

First suppose that $k\geq0$, then $|U_1^+|\geq n$ and $\delta(\red{G}[U_1^+])\geq|U_1^+|-2\xi n$. If $T$ contains at least $n/100$ vertex-disjoint bare paths of length 4, then we can find a red copy of $T$ in $G[U_1^+]$ by Lemma~\ref{lemma:bare-paths}. Otherwise, $T$ contains at least $n/20$ leaves by Lemma~\ref{lemma:paths-leaf}. Then, we can find a red copy of $T$ in $G[U_1^+]$ by Lemma~\ref{lemma:leaves}.

Now suppose $k<0$, so $|U_2^+|\geq\tau_2$. Note that $\nu\geq t_1\geq n/2$ and $|U_1|\geq(1-\mu)(\tau_1+t_2)\geq\tau_1+n/5$ from assumption. Thus, we can find a blue copy of $\cT$ in $G[U_1,U_2^+]$ by Lemma~\ref{lemma:bipartite:anyleaves}.
\end{proof}

\subsection{Type 2}\label{sec:type2}
\begin{lemma}[\textbf{Type 2}]\label{lemma:type2}
There exists a constant $c>0$ such that the following holds. Let $1/n\leq1/\nu\ll\mu\leq1$. Let $T$ be an $n$-vertex tree with $\Delta(T)\leq cn/\log n$ and bipartition class sizes $t_1\geq t_2\geq(t_1-2)/3$. Let $\cT$ be a $\nu$-vertex tree with $\Delta(\cT)\leq c\nu/\log\nu$ and bipartition class sizes $\tau_1\geq\tau_2\geq(t_1-2)/2$. Suppose also that $\nu\geq(1-\mu)(t_1+\tau_2)$, $\tau_2\geq t_2$, and $\tau_2+t_2\geq t_1-1$.

Let $G$ be a red/blue coloured complete graph with $\max\{n+\tau_2,2t_1\}-1$ vertices which is Type 2 $(\mu,T,\cT)$-extremal. That is, $G$ contains disjoint $U_1,U_2\subset V(G)$ such that $|U_1|\geq(1-\mu)\nu$, $|U_2|\geq(1-\mu)t_2$, $\red{d}(v,U_1)\leq\mu n$ for every $v\in U_1$, and $\blue{d}(v,U_{3-i})\leq\mu n$ for every $i\in[2]$ and $v\in U_i$. Then, $G$ contains either a red copy of $T$ or a blue copy of $\cT$.
\end{lemma}
\begin{proof}
Let $\xi=1/100$, and let $c$ be smaller than the constants given by applying Lemma~\ref{lemma:leaves} and Lemma~\ref{lemma:bipartite:anyleaves} with $\xi$. Let $U_1^+=\{v\in V(G):\blue{d}(v,U_1)\geq|U_1|-\xi n\}$ and $U_2^+=V(G)\setminus U_1^+$. Then, $\red{d}(v,U_1)\geq\xi n$ for every $v\in U_2^+$. Note also that $U_1\subset U_1^+$ and $U_2\subset U_2^+$. Let $k=|U_1^+|-\nu$, so $|k|\leq10\mu n$ and $|U_2^+|\geq t_2-k-1$.

First suppose that $k\geq0$, then $|U_1^+|\geq\nu$ and $\delta(\blue{G}[U_1^+])\geq|U_1^+|-2\xi n$. If $\cT$ contains at least $\nu/100$ vertex-disjoint bare paths of length 4, then we can find a blue copy of $\cT$ in $G[U_1^+]$ by Lemma~\ref{lemma:bare-paths}. Otherwise, $\cT$ contains at least $\nu/20$ leaves by Lemma~\ref{lemma:paths-leaf}. Then, we can find a blue copy of $\cT$ in $G[U_1^+]$ by Lemma~\ref{lemma:leaves}.

Now suppose $k<0$, so $|U_2^+|\geq t_2$. Note that $|U_1|\geq(1-\mu)^2(t_1+\tau_2)\geq t_1+n/5$ from assumption. Thus, we can find a red copy of $T$ in $G[U_1,U_2^+]$ by Lemma~\ref{lemma:bipartite:anyleaves}.
\end{proof}

\subsection{Type 3}\label{sec:type3}
\begin{lemma}[\textbf{Type 3}]\label{lemma:type3}
There exists a constant $c>0$ such that the following holds. Let $1/n\leq1/\nu\ll\mu\leq1$. Let $T$ be an $n$-vertex tree with $\Delta(T)\leq cn/\log n$ and bipartition class sizes $t_1\geq t_2\geq(t_1-2)/3$. Let $\cT$ be a $\nu$-vertex tree with $\Delta(\cT)\leq c\nu/\log\nu$ and bipartition class sizes $\tau_1\geq\tau_2\geq(t_1-2)/2$. Suppose also that $\nu\geq t_1\geq(1-\mu)(t_2+\tau_2)\geq(1-\mu)(t_1-1)$ and $\tau_2\geq t_2$.

Let $G$ be a red/blue coloured complete graph with $\max\{n+\tau_2,2t_1\}-1$ vertices which is Type 3 $(\mu,T,\cT)$-extremal. That is, $G$ contains disjoint $U_1,U_2\subset V(G)$ such that $|U_1|,|U_2|\geq(1-\mu)t_1$, $\red{d}(v,U_i)\leq\mu n$ for every $i\in[2]$ and every $v\in U_i$, and $\blue{d}(v,U_{3-i})\leq\mu n$ for every $i\in[2]$ and $v\in U_i$. Then, $G$ contains either a red copy of $T$ or a blue copy of $\cT$.
\end{lemma}
\begin{proof}
Let $c$ be smaller than the constants given by applying Lemma~\ref{lemma:leaves} and Lemma~\ref{lemma:bipartite:anyleaves} with $\xi=1/100$. Let $U_1^+,U_2^+\subset V(G)$ be maximal disjoint sets with $U_1\subset U_1^+$, $U_2\subset U_2^+$, and $\red{d}(v,U_{3-i})\geq|U_{3-i}|/3$ for every $i\in [2]$ and $v\in U_i^+$. Note that $|U_i^+\setminus U_i|\leq10\mu n$ for each $i\in[2]$. By relabelling if necessary, we can assume that $|U_1^+|\geq |U_2^+|$.  

First, suppose that there are two distinct vertices $v_1,v_2\in V(G)\setminus (U_1^+\cup U_2^+)$. By the maximality of $U_1^+$ and $U_2^+$, we must have $\blue{d}(v_i,U_j)\geq2|U_j|/3$ for each $i,j\in [2]$. As $t_1\geq(1-\mu)(t_2+\tau_2)\geq(2-2\mu)t_2$, we have $|U_i|\geq t_1-\mu n\geq (2/3-2\mu)n\geq(2/3-2\mu)\nu$ for each $i\in [2]$. By Lemma~\ref{lemma:splittreewith2vertices}, there is a partition $V(\cT)=A\cup B$ with $|A|,|B|\leq (2/3-10\mu)\nu$, such that $\cT[A]$ is a tree and $A':=\{v\in A:d_{\cT}(v,B)>0\}$ is an independent set in $\cT$ with $|A'|\leq 2$. View $\cT[A]$ as a tree rooted at a vertex in $A'$, and embed this vertex to $v_1$. Then, greedily extend this to an embedding of the tree $\cT[A]$ in $\blue{G}[U_1\cup\{v_1,v_2\}]$, such that if there is another vertex in $A'$, then it is embedded to $v_2$. This is possible as $|U_1|-|A|\gg\mu n$ and $|\blue{N}(v_1,U_1)\cap\blue{N}(v_2,U_1)|\geq|U_1|/3\gg\mu n$. We can then extend this to a copy of $\cT$ by greedily embedding $\cT[A'\cup B]$ in $\blue{G}[U_2\cup\{v_1,v_2\}]$. 

Therefore, we may assume from now on that $|V(G)\setminus (U_1^+\cup U_2^+)|\leq1$. Note that if $n+\tau_2\geq2t_1$, then $3\tau_2\geq(t_2+\tau_2)+(2\tau_2-t_2)\geq t_1+t_2=n$. It follows that $|U_1^+|\geq\ceil{(|G|-1)/2}\geq\max\{\ceil{2n/3}-1,t_1-1\}$.

If there exists $v\in V(G)$ with at least $20\mu n$ red neighbours in both $U_1^+$ and $U_2^+$, then $v$ has at least $20\mu n-10\mu n\geq\Delta(T)$ red neighbours in both $U_1$ and $U_2$. By Lemma~\ref{lemma:bipartitesplittree}, we can find subtrees $T_1$ and $T_2$ of $T$ with a unique common vertex $t$, such that $V(T_1)\cup V(T_2)=V(T)$, and $10\mu n\leq|V(T_1)\cap V_1|-|V(T_1)\cap V_2|\leq25\mu n$. Embed $t$ to $v$, then we can greedily embed both $T_1$ and $T_2$ so that vertices in $V(T_2)\cap V_1$ and $V(T_1)\cap V_2$ go into $U_1$ and vertices in $V(T_1)\cap V_1$ and $V(T_2)\cap V_2$ go into $U_2$. This is possible because \[|V(T_1)\cap V_2|+|V(T_2)\cap V_1|\leq|V(T_1)\cap V_1|+|V(T_2)\cap V_1|-10\mu n\leq t_1+1-10\mu n\leq|U_1|-\mu n,\]
\[|V(T_1)\cap V_1|+|V(T_2)\cap V_2|\leq|V(T_1)\cap V_2|+|V(T_2)\cap V_2|+25\mu n\leq t_2+1+25\mu n\leq |U_2|-\mu n,\]
and $v$ has at least $\Delta(T)$ red neighbours in both $U_1$ and $U_2$. Thus, we can assume that no such $v$ exists, so in particular $\delta(\blue{G}[U_i^+])\geq|U_i^+|-20\mu n$ for each $i\in [2]$.

Now suppose that $V(G)\setminus (U_1^+\cup U_2^+)$ contains a unique vertex $w$. By maximality, $\blue{d}(w,U_j^+)\geq\blue{d}(w,U_j)\geq2|U_j|/3$ for each $j\in [2]$. 
By Lemma~\ref{lemma:splittree}, there exist subtrees $\cT_1$ and $\cT_2$ of $\cT$ with a unique common vertex $t$, such that $V(\cT_1)\cup V(\cT_2)=V(\cT)$, and $\ceil{\nu/3}\leq|\cT_1|\leq|\cT_2|\leq\ceil{2\nu/3}$. For each $j\in[2]$, $\cT_j$ either has $\nu/100$ vertex-disjoint bare paths of length 4 or $\nu/100$ leaves by Lemma~\ref{lemma:paths-leaf}, so we can use Lemma~\ref{lemma:bare-paths} or Lemma~\ref{lemma:leaves}, respectively, to embed $\cT_j$ into $\blue{G}[U_j^+\cup\{w\}]$, with $t$ embedded to $w$. Together, this gives a blue copy of $\cT$ in $G$.

Finally, suppose that $V(G)\setminus (U_1^+\cup U_2^+)=\varnothing$. Then, $|U_1^+|\geq t_1$, and $|U_2|\geq(1-\mu)t_1\geq t_2+n/5$, so we can find a red copy of $T$ in $G[U_1^+,U_2]$ by Lemma~\ref{lemma:bipartite:anyleaves}. 
\end{proof}

\subsection{Type 4}\label{sec:type4}
\begin{lemma}[\textbf{Type 4}]\label{lemma:type4}
There exists a constant $c>0$ such that the following holds. Let $1/n\leq1/\nu\ll\mu\leq1$. Let $T$ be an $n$-vertex tree with $\Delta(T)\leq cn/\log n$ and bipartition class sizes $t_1\geq t_2\geq(t_1-2)/3$. Let $\cT$ be a $\nu$-vertex tree with $\Delta(\cT)\leq c\nu/\log\nu$ and bipartition class sizes $\tau_1\geq\tau_2\geq(t_1-2)/2$. Suppose also that $\tau_1\geq(1-\mu)(n+\tau_2)/2$, $\nu\geq t_1\geq s_1$, $t_2+\tau_2\geq t_1-1$, and $\tau_2\geq t_2$.

Let $G$ be a red/blue coloured complete graph with $\max\{n+\tau_2,2t_1\}-1$ vertices which is Type 4 $(\mu,T,\cT)$-extremal. That is, $G$ contains disjoint $U_1,U_2\subset V(G)$ such that $|U_1|,|U_2|\geq(1-\mu)\tau_1$, $\blue{d}(v,U_i)\leq\mu n$ for every $i\in[2]$ and every $v\in U_i$, and $\red{d}(v,U_{3-i})\leq\mu n$ for every $i\in[2]$ and $v\in U_i$. Then, $G$ contains either a red copy of $T$ or a blue copy of $\cT$.
\end{lemma}
\begin{proof}
Let $c$ be smaller than the constants given by applying Lemma~\ref{lemma:leaves} and Lemma~\ref{lemma:bipartite:anyleaves} with $\xi=1/100$. From assumptions, $2\tau_1\geq(1-\mu)(\tau_1+t_2+\tau_2)$, so $\tau_1\geq(1-2\mu)(t_2+\tau_2)\geq(1-3\mu)t_1$. Also, $2\tau_1\geq(1-\mu)(\nu+\tau_2)$, so $t_1\geq\tau_1\geq(2-4\mu)\tau_2\geq(2-4\mu)t_2$. Let $U_1^+,U_2^+\subset V(G)$ be maximal disjoint sets with $U_1\subset U_1^+$, $U_2\subset U_2^+$, and $\blue{d}(v,U_{3-i})\geq|U_{3-i}|/3$ for every $i\in [2]$ and $v\in U_i^+$. Note that $|U_i^+\setminus U_i|\leq10\mu n$ for each $i\in[2]$. By relabelling if necessary, we can assume that $|U_1^+|\geq |U_2^+|$.  

First, suppose that there are two distinct vertices $v_1,v_2\in V(G)\setminus (U_1^+\cup U_2^+)$. By the maximality of $U_1^+$ and $U_2^+$, we must have $\red{d}(v_i,U_j)\geq2|U_j|/3$ for each $i,j\in [2]$. As $\tau_1\geq(1-3\mu)t_1$, we have $|U_i|\geq(1-\mu)\tau_1\geq (2/3-10\mu)n$ for each $i\in [2]$. By Lemma~\ref{lemma:splittreewith2vertices}, there is a partition $V(T)=A\cup B$ with $|A|,|B|\leq (2/3-20\mu)n$, such that $T[A]$ is a tree and $A':=\{v\in A:d_{T}(v,B)>0\}$ is an independent set in $T$ with $|A'|\leq 2$. Then, as in the proof of Lemma~\ref{lemma:type3}, we can use this partition to find a red copy of $T$ in $G$, with any vertex in $A'$ embedded to $\{v_1,v_2\}$.

Therefore, we may assume from now on that $|V(G)\setminus (U_1^+\cup U_2^+)|\leq1$. Note that $|U_1^+|\geq\ceil{(|G|-1)/2}\geq\max\{\ceil{2n/3}-1,t_1-1\}$.

If there exists $v\in V(G)$ with at least $20\mu n$ blue neighbours in both $U_1^+$ and $U_2^+$, then like in the proof of Lemma~\ref{lemma:type3}, we can use Lemma~\ref{lemma:bipartitesplittree} to find a blue copy of $\cT$ in $G$. Thus, we can assume that no such $v$ exists, so in particular $\delta(\blue{G}[U_i^+])\geq|U_i^+|-20\mu n$ for each $i\in [2]$.

Now suppose that $V(G)\setminus (U_1^+\cup U_2^+)$ contains a unique vertex $w$. By maximality, $\red{d}(w,U_j^+)\geq\red{d}(w,U_j)\geq2|U_j|/3$ for each $j\in [2]$. 
By Lemma~\ref{lemma:splittree}, there exist subtrees $T_1$ and $T_2$ decomposing $T$, such that $\ceil{n/3}\leq|T_1|\leq|T_2|\leq\ceil{2n/3}$. Then, for each $j\in[2]$, we can use either Lemma~\ref{lemma:bare-paths} or Lemma~\ref{lemma:leaves} to embed $T_j$ into $\red{G}[U_j^+\cup\{w\}]$, with the unique common vertex of $T_1$ and $T_2$ embedded to $w$. Together, this gives a red copy of $T$ in $G$.

Finally, suppose that $V(G)\setminus (U_1^+\cup U_2^+)=\varnothing$. Then, $|U_1^+|\geq t_1\geq\tau_1$, and $|U_2|\geq(1-\mu)\tau_1\geq\tau_2+n/5$, so we can find a blue copy of $\cT$ in $G[U_1^+,U_2]$ by Lemma~\ref{lemma:bipartite:anyleaves}. 
\end{proof}

\section{Counterexamples}\label{sec:counter}
In this section, we construct examples proving Theorem~\ref{thm:counterex1} and Theorem~\ref{thm:counterex2}, thus showing that neither $\tau_2\geq t_2$ nor $\nu\geq t_1$ can be removed from the list of assumptions in Theorem~\ref{thm:main}.

The two constructions share a similar idea, and can be viewed as a non-bipartite version and a bipartite version. We now give a sketch for the non-bipartite version, used for Theorem~\ref{thm:counterex1}. As a warm-up, let's suppose we want to find a tree $T$ with bipartition class sizes $t_1\geq t_2$ and a tree $\cT$ with bipartition class sizes $t_1\geq\tau_2$, with $t_1=t_2+\tau_2$, such that $R(T,\cT)\geq2t_1=\underline{R}(T,\cT)+1$. Suppose the host graph $G$ is quite close to the extremal construction~\ref{low:3}, and has a vertex partition $V(G)=U_1\cup U_2\cup\{w\}$, such that $|U_1|=|U_2|=t_1-1$, all of $\red{G}[U_1],\red{G}[U_2]$ and $\red{G}[U_1\cup U_2,\{w\}]$ are complete, and all other edges are blue. Then, $G$ contains no blue copy of $\cT$ as the larger bipartition class of $\cT$ has size $t_1>|U_1|,|U_2|$. To ensure that $G$ contains no red copy of $T$, let $T$ be a tree containing a vertex $v$ such that $T-v$ consists of exactly 3 components, each of size $(t_1+t_2-1)/3$. If $t_1\leq2t_2-2$, then $(t_1+t_2-1)/3\geq t_1/2$, so each of $U_1$ and $U_2$ is not big enough to accommodate 2 of these 3 components. However, a simple case analysis on where $v$ is embedded to shows that we cannot avoid embedding 2 of these 3 components to the same side, so no red copy of $T$ exists in $G$. 

This idea can be extended to find arbitrarily large trees $T$ and $\cT$ satisfying $R(T,\cT)\geq\underline{R}(T,\cT)+C$ for any $C\geq1$. Let $G$ be the same graph as above, except now $\{w\}$ is replaced by a set $W$ of size $C$. Let $T$ be the tree obtained by attaching a tree of size $r\gg C$ to each of the $3^C$ leaves of a perfect ternary tree $T'$ with $C+1$ levels. Intuitively, if a red embedding attempt of $T$ starts in $U_1$, then the $C$ vertices in $W$ can be used to move at most $3^{C-1}+3^{C-2}+\cdots+1=(3^C-1)/2$ trees attached to leaves of $T'$ to $U_2$, so at least $(3^C+1)/2$ trees attached to leaves of $T'$ must be embedded into $U_1$. This is proved more rigorously in Lemma~\ref{lem:counterex1} below. Therefore, if we choose the parameters carefully so that $|U_1|=|U_2|=t_1-1<(3^C+1)r/2$, then no red copy of $T$ exists. The full proof will be carried out in Section~\ref{sec:counter1}.

For the bipartite version of the construction, used for Theorem~\ref{thm:counterex2}, let $G$ be a red/blue coloured complete graph with a vertex partition $V(G)=U_1\cup U_2\cup W$, such that $|U_1|=|U_2|=|\cT|-1$, $|W|=C-1$, $\red{G}[U_1,U_2]$ and $\red{G}[U_1\cup U_2,W]$ are complete, and all other edges are blue. Then, it is clear that $G$ contains no blue copy of $\cT$. Let $T$ be constructed by gluing the same tree $\overline{T}$ to each leaf of a perfect ternary tree $T'$ with $C+1$ levels. Similar to above, the rough idea (see Lemma~\ref{lem:counterex2}) is that the $C-1$ vertices in $W$ can only be used to flip the embedding direction of slightly less than half of the copies of $\overline{T}$ attached to the leaves of $T'$. Thus, there is a lower bound on $|U_1|=|\cT|-1$ if a red copy of $T$ exists in $G$, and we can rule this out with careful choices of parameters. This will be carried out in Section~\ref{sec:counter2}

\subsection{Proof of Theorem~\ref{thm:counterex1}}\label{sec:counter1}
\begin{lemma}\label{lem:counterex1}
Let $r\gg C\geq 1$ be integers. Let $T'$ be the perfect ternary tree with $C+1$ levels, $\sum_{i=0}^{C-1}3^i=(3^C-1)/2$ internal vertices, and $3^C$ leaves. Let $T$ be a tree obtained by gluing to each leaf of $T'$ a (possibly different) tree of size $r$.

Let $G$ be a graph with a vertex partition $V(G)=U_1\cup U_2\cup W$, such that the following hold.
\begin{itemize}
    \item $|U_1|=|U_2|$ and $|W|=C$.
    \item $G[U_1]$, $G[U_2]$, and $G[U_1\cup U_2,W]$ are all complete, and there is no other edge in $G$.
\end{itemize}
If there is a copy of $T$ in $G$, then $|U_1|\geq(3^C+1)r/2$.
\end{lemma}
\begin{proof}
Let $\psi:V(T)\to V(G)$ be an embedding of $T$ in $G$, and suppose for a contradiction that $|U_1|=|U_2|<(3^C+1)r/2$. Since there is no edge between $U_1$ and $U_2$, every component of $T-\psi^{-1}(W)$ is embedded entirely within one of $U_1$ and $U_2$. In particular, $\psi^{-1}(W)\not=\varnothing$. 

For each $0\leq i\leq C-1$, let $L_i$ be the set of vertices in the $i$-th level of $T'$, so that $|L_i|=3^i$. Let $L_C=V(T)\setminus(\cup_{i=0}^{C-1}L_i)$. Let $\ell_i=|\psi^{-1}(W)\cap\left(\cup_{j=0}^{i}L_j\right)|$ for every $0\leq i\leq C$, and note that $\ell_C=C$.

Suppose first that $\ell_i\leq i$ for every $0\leq i\leq C$. In particular, $\ell_0=0$, so we may assume without loss of generality that the root of $T'$ is embedded into $U_1$. Then, without the help of vertices in $W$, each of the $3^C$ trees of size $r$ glued to the leaves of $T'$ will be embedded into $U_1$. Since $\ell_i\leq i$ for every $i\in[C]$, the $C$ vertices in $W$ can be used to transfer at most $\sum_{i=1}^C3^{C-i}=(3^C-1)/2$ trees glued to the leaves of $T'$ to $U_2$, so at least $(3^C+1)/2$ of them still need to be embedded inside $U_1$, contradicting $|U_1|<(3^C+1)r/2$. 


Therefore, there exists $0\leq i\leq C$ satisfying $\ell_i\geq i+1$. Let $0\leq i^*\leq C$ be maximal subject to $\ell_{i^*}\geq i^*+1$, and note that $i^*\leq C-1$. 
Since $i^*$ is maximal, $\ell_{i^*+1}\leq i^*+1$, so no vertex in $L_{i^*+1}$ is embedded into $W$. Thus, by pigeonhole, and without loss of generality, we can assume that a set $L_{i^*+1}'$ of at least $(3^{i^*+1}+1)/2$ vertices in $L_{i^*+1}$ are embedded into $U_1$. Let $\mathcal{T}$ be the set of at least $(3^{i^*+1}+1)3^{C-i^*-1}/2$ trees of size $r$ glued to the leaves of $T'$ that are descendants of vertices in $L_{i^*+1}'$. Without the help of vertices in $W$, every tree in $\mathcal{T}$ will be embedded into $U_1$. By the maximality of $i^*$ again, $\ell_i\leq i$ for every $i\geq i^*+1$, so there are at most $i-i^*-1$ vertices in $\cup_{j=i^*+1}^iL_j$ that are embedded into $W$. Together, they can be used to transfer at most $\sum_{j=i^*+2}^C3^{C-j}=(3^{C-i^*-1}-1)/2$ trees in $\mathcal{T}$ to $U_2$. Thus, at least \[\frac12(3^{i^*+1}+1)3^{C-i^*-1}-\frac12(3^{C-i^*-1}-1)=\frac12(3^{C}+1)\]
trees in $\mathcal{T}$ still need to be embedded into $U_1$, contradicting $|U_1|<(3^C+1)r/2$.
\end{proof}



\begin{proof}[Proof of Theorem~\ref{thm:counterex1}]
Since $r\gg C$, $(3^C+1)r/(2r-(3^C-1)/2)\leq 3^C+1$, so by Lemma~\ref{lemma:maxdegbound}, we can find a tree $\cT$ satisfying both~\ref{construct:1} and~\ref{construct:3}.

Let $T'$ be the perfect ternary tree with $C+1$ levels, $\sum_{i=0}^{C-1}3^i=(3^C-1)/2$ internal vertices, and $3^C$ leaves. Let $T$ be a tree obtained by gluing to each leaf of $T'$ a (possibly different) tree of size $2r$. Then, $|T|=3^C\cdot 2r+(3^C-1)/2$.
Moreover, by picking the trees glued to the leaves of $T'$ carefully, we can ensure that both~\ref{construct:1} and~\ref{construct:2} hold, as~\ref{construct:2} is only asking for mildly unbalanced bipartition classes.

Let $G$ be the red/blue coloured complete graph on $2(3^C+1)r+C-2$ vertices with a vertex partition $V(G)=U_1\cup U_2\cup W$, such that the following conditions hold.
\begin{itemize}
    \item $|U_1|=|U_2|=(3^C+1)r-1$ and $|W|=C$.
    \item $\red{G}[U_1]$, $\red{G}[U_2]$, and $\red{G}[U_1\cup U_2,W]$ are all complete. 
    \item $\blue{G}[U_1,U_2]$ and $\blue{G}[W]$ are both complete.
\end{itemize}

Note that $\blue{G}$ contains exactly two connected components with vertex sets $U_1\cup U_2$ and $W$, respectively. Neither can contain a copy of $\cT$, as $\blue{G}[W]$ is too small, while neither side of the bipartite graph $\blue{G}[U_1,U_2]$ is large enough to accommodate the larger bipartition class of $\cT$. Meanwhile, $\red{G}$ contains no copy of $T$ by Lemma~\ref{lem:counterex1}, as 
\[\frac12(3^C+1)\cdot2r=(3^C+1)r>|U_1|.\]
Therefore,~\ref{construct:4} holds, which completes the proof. 
\end{proof}

\subsection{Proof of Theorem~\ref{thm:counterex2}}\label{sec:counter2}
\begin{lemma}\label{lem:counterex2}
Let $r\gg C\geq 2$ and $\rho\geq2$ be integers. Let $T'$ be the perfect ternary tree with $C+1$ levels, $\sum_{i=0}^{C-1}3^i=(3^C-1)/2$ internal vertices, and $3^C$ leaves. Let $T$ be a tree obtained by gluing to each leaf of $T'$ the same tree $\overline{T}$ with bipartition class sizes $\rho r$ and $r$.

Let $G$ be a graph with a vertex partition $V(G)=U_1\cup U_2\cup W$, such that the following hold.
\begin{itemize}
    \item $|U_1|=|U_2|$, and $|W|=C-1$.
    \item $G[U_1,U_2]$ and $G[U_1\cup U_2,W]$ are both complete, and there is no other edge in $G$.
\end{itemize}
If there is a copy of $T$ in $G$, then $|U_1|\geq((3^C+3)\rho r+(3^C-3)r)/2$.
\end{lemma}
\begin{proof}
Let $\psi:V(T)\to V(G)$ be an embedding of $T$ in $G$, and suppose for a contradiction that $|U_1|<((3^C+3)\rho r+(3^C-3)r)/2$. Since both $G[U_1]$ and $G[U_2]$ are empty, every component of $T-\psi^{-1}(W)$ has one of its bipartition classes embedded into $U_1$, and the other into $U_2$. In particular, $\psi^{-1}(W)\not=\varnothing$ as $|U_1|<3^C\rho r$. 

For each $0\leq i\leq C-1$, let $L_i$ be the set of vertices in the $i$-th level of $T'$, so that $|L_i|=3^i$. Let $L_C=V(T)\setminus(\cup_{i=0}^{C-1}L_i)$. Let $\ell_i=|\psi^{-1}(W)\cap\left(\cup_{j=0}^{i}L_j\right)|$ for every $0\leq i\leq C$, and note that $\ell_C\leq C-1$.

Suppose first that $\ell_i\leq i$ for every $0\leq i\leq C$. In particular, $\ell_0=0$, so we may assume without loss of generality that the root of $T'$ is embedded into $U_1$ if it is in the larger bipartition class of $T$, and into $U_2$ otherwise. Then, without the help of vertices in $W$, each of the $3^C$ copies of $\overline{T}$ glued to the leaves of $T'$ will have its larger bipartition class embedded into $U_1$, and its smaller class embedded into $U_2$. Note that if a vertex in $L_C$ is embedded into $W$, then it can be used to potentially embed all vertices in its corresponding copy of $\overline{T}$ except itself into $U_2$. This can move at most $\rho r$ vertices that would be embedded into $U_1$ to $U_2$. However, if a vertex in $L_{C-1}$ is embedded into $W$, then it can flip the embedding direction of all 3 copies of $\overline{T}$ below it in $T$, reducing the number of vertices embedded into $U_1$ by $3(\rho-1)r>\rho r$. Thus, we may assume all vertices embedded into $W$ are from $\cup_{i=0}^{C-1}L_i$. Then, since $\ell_i\leq i$ for every $0\leq i\leq C$, the $C-1$ vertices in $W$ can be used to flip the embedding direction of at most $\sum_{i=1}^{C-1}3^{C-i}=(3^C-3)/2$ copies of $\overline{T}$ attached to the leaves of $T'$. Thus, the number of vertices embedded into $U_1$ is still at least $((3^C+3)\rho r+(3^C-3)r)/2>|U_1|$, a contradiction.


Therefore, we may assume that $\ell_i\geq i+1$ for some $0\leq i\leq C$. Let $0\leq i^*\leq C$ be maximal subject to $\ell_{i^*}\geq i^*+1$, and note that $i^*\leq C-2$. Since $i^*$ is maximal, $\ell_{i^*+1}\leq i^*+1$, so every vertex in $L_{i^*+1}$ is not embedded into $W$. Thus, by pigeonhole, and without loss of generality, there is a set $L_{i^*+1}'$ of at least $(3^{i^*+1}+1)/2$ vertices in $L_{i^*+1}$ which are embedded into $U_1$ if vertices in $L_{i^*+1}$ are in the larger bipartition class of $T$, and into $U_2$ otherwise. Let $\mathcal{T}$ be the set of at least $(3^{i^*+1}+1)3^{C-i^*-1}/2$ copies of $\overline{T}$ which are glued to the descendants of $L_{i^*+1}'$. Without the help of vertices in $W$, every tree in $\mathcal{T}$ will have its larger bipartition class embedded into $U_1$, and its smaller class embedded into $U_2$. By the maximality of $i^*$ again, $\ell_i\leq i$ for every $i^*+1\leq i\leq C-1$, so there are at most $i-i^*-1$ vertices in $\cup_{j=i^*+1}^iL_j$ that are embedded into $W$. Together, they can be used to flip the embedding direction of at most $\sum_{j=i^*+2}^{C-1}3^{C-j}=(3^{C-i^*-1}-3)/2$ trees in $\mathcal{T}$. Thus, at least 
\[\frac12((3^{i^*+1}+1)3^{C-i^*-1}\rho r+(3^{i^*+1}-1)3^{C-i^*-1}r)-\frac12(3^{C-i^*-1}-3)(\rho r-r)=\frac12(3^C+3)\rho r+(3^C-3)r\]
vertices are still embedded into $U_1$, a contradiction.
\end{proof}

\begin{proof}[Proof of Theorem~\ref{thm:counterex2}]
Let $T'$ be the perfect ternary tree with $C+1$ levels, $\sum_{i=0}^{C-1}3^i=(3^C-1)/2$ internal vertices, and $3^C$ leaves. Let $T$ be a tree obtained by gluing to each leaf of $T'$ the same tree $\overline{T}$ with bipartition class sizes $\rho r$ and $r$. Then, it is easy to see that~\ref{construct:2:2} holds. Moreover, by Lemma~\ref{lemma:maxdegbound}, we can pick $\overline{T}$ so that $\Delta(\overline{T})\leq\rho+2$, and thus $\Delta(T)\leq\rho+3$ and~\ref{construct:2:1} holds for $T$. 

Let $\tau_1=t_1+t_2-((3^C+3)\rho r+(3^C-3)r)/2=((3^C-3)\rho r+(3^C+3)r)/2+O_C(1)$, and let $\tau_2=t_1+t_2-2\tau_1$. Note that 
\begin{align*}
\tau_2&=(3^C+3)\rho r+(3^C-3)r-t_1-t_2\\
&=(3^C+3)\rho r+(3^C-3)r-3^C\rho r-3^Cr-O_C(1)=3(\rho-1)r-O_C(1).
\end{align*}
Since $\tau_1/\tau_2\leq t_1/\tau_2\leq(3^C\rho r+O_C(1))/(3(\rho-1)r-O_C(1))\leq 3^{C+1}$, we can use Lemma~\ref{lemma:maxdegbound} to find a tree $\cT$ with bipartition class sizes $\tau_1$ and $\tau_2$, such that both~\ref{construct:2:1} and~\ref{construct:2:3} hold. Also, note that $2\nu=2\tau_1+\tau_2+\tau_2=t_1+t_2+\tau_2$ from the definition of $\tau_2$, and $\nu=\tau_1+\tau_2=t_1+t_2-\tau_1=((3^C+3)\rho r+(3^C-3)r)/2<t_1$, so $\underline{R}(T,\cT)=2\nu-1$.

Let $G$ be the red/blue coloured complete graph on $2\nu+C-3$ vertices with a vertex partition $V(G)=U_1\cup U_2\cup W$, such that the following conditions hold.
\begin{itemize}
    \item $|U_1|=|U_2|=\nu-1$ and $|W|=C-1$.
    \item $\red{G}[U_1,U_2]$ and $\red{G}[U_1\cup U_2,W]$ are both complete. 
    \item $\blue{G}[U_1]$, $\blue{G}[U_2]$, and $\blue{G}[W]$ are all complete.
\end{itemize}

Note that $\blue{G}$ contains exactly three connected components with vertex sets $U_1$, $U_2$, and $W$, respectively. None of them is large enough to contain a blue copy of $\cT$. Meanwhile, $\red{G}$ contains no copy of $T$ by Lemma~\ref{lem:counterex2}, as
\[\frac12((3^C+3)\rho r+(3^C-3)r)=\nu>|U_1|.\]
Therefore,~\ref{construct:2:4} holds, which completes the proof. 
\end{proof}

\section{Concluding remarks}\label{sec:conclude}
In this section, we discuss some related results and pose some open questions. 

The most natural open question is whether the sublinear maximum degree condition in Theorem~\ref{thm:main} can be improved to a linear one like in Theorem~\ref{thm:oldmain}. The extremal part of the proof of Theorem~\ref{thm:oldmain} crucially uses $T=\cT$, as the failure to embed $T$ in one colour implies that within the host graph $G$, there exists a certain structure in the other colour whose size depends on $T$, which can be used to embed $T$ in the other colour. The same proof no longer works in general when $T\not=\cT$. 
\begin{question}\label{question:1}
Can the conditions $\Delta(T)\leq cn/\log n$ and $\Delta(\cT)\leq c\nu/\log\nu$ in Theorem~\ref{thm:main} be replaced with $\Delta(T)\leq cn$ and $\Delta(\cT)\leq c\nu$?
\end{question}

Recall that Theorem~\ref{thm:stability}, which is the stability part of our proof of Theorem~\ref{thm:main}, is stated with the conditions $\Delta(T)\leq cn$ and $\Delta(\cT)\leq c\nu$. If we are only aiming for an approximate version of Theorem~\ref{thm:main}, then the answer to Question~\ref{question:1} is yes. This can be shown by starting with a slightly larger red/blue coloured complete graph $G$, then following the proof of Theorem~\ref{thm:stability} until eventually in \textbf{Stage 4} noting that being $(\mu,T,\cT)$-extremal is no longer a possible outcome. More precisely, we have the following. 
\begin{theorem}\label{thm:main:approx}
For every $\mu>0$, there exists $c=c_\mu>0$ such that the following hold. 

Let $n\geq\nu$ be sufficiently large. Let $T$ be an $n$-vertex tree with $\Delta(T)\leq cn$ and bipartition class sizes $t_1\geq t_2$. Let $\cT$ be a $\nu$-vertex tree with $\Delta(\cT)\leq c\nu$ and bipartition class sizes $\tau_1\geq\tau_2$. Suppose $\tau_2\geq t_2$ and $\nu\geq t_1$, then \[R(T,\cT)\leq(1+\mu)\underline{R}(T,\cT)\leq(1+\mu)\max\{n+\tau_2,2t_1\}.\]
\end{theorem}

Theorem~\ref{thm:counterex1} and Theorem~\ref{thm:counterex2} show that if one of the assumptions $\tau_2\geq t_2$ and $\nu\geq t_1$ is dropped from Theorem~\ref{thm:main}, then $R(T,\cT)$ can exceed $\underline{R}(T,\cT)$ by an arbitrary additive factor. Another natural question is whether the lower bound $\underline{R}(T,\cT)$ is still approximately tight if either of these assumptions is dropped. 
\begin{question}\label{question:2}
Does Theorem~\ref{thm:main:approx} hold, possibly with a stronger maximum degree condition, if one of the assumptions $\tau_2\geq t_2$ and $\nu\geq t_1$ is dropped?
\end{question}

Recall that when $T=\cT$ is a double star, then by~\cite{NSZ}, $R(T)=R(T,T)$ can exceed $\underline{R}(T,T)$ by a multiplicative factor. The following construction provides examples of pairs $(T,\cT)$ of distinct trees, such that $R(T,\cT)$ exceed $R(T,\cT)$ by a multiplicative factor. This is related to both questions above, though does not fully answer either. In regards to Question~\ref{question:1}, the example shows that if the condition $\tau_2\geq t_2$ is dropped from Theorem~\ref{thm:main}, then any linear maximum degree condition is not sufficient. Regarding Question~\ref{question:2}, it shows that Theorem~\ref{thm:main:approx} is false if the condition $\tau_2\geq t_2$ is dropped and the order of the quantifiers for $\mu$ and $c$ is switched.
\begin{theorem}\label{thm:smallrandom}
For every $c>0$, there exists $0<\mu\ll c$, such that for all sufficiently large $n$, there exist two trees $T$ and $\cT$ satisfying the following. 
\stepcounter{propcounter}
\begin{enumerate}[label = \emph{\textbf{\Alph{propcounter}\arabic{enumi}}}]
    \item\labelinthm{construct:3:1}  $\Delta(T)\leq cn$ and $\Delta(\cT)=cn$.
    \item\labelinthm{construct:3:2}  $T$ has bipartition class sizes $2n$ and $2n-\mu n$.
    \item\labelinthm{construct:3:3}  $\cT$ has bipartition class sizes $2n$ and $\mu n$.
    \item\labelinthm{construct:3:4}  $R(T,\cT)\geq4n+\mu n\geq\underline{R}(T,\cT)+\mu n$.
\end{enumerate}
\end{theorem}
\begin{proof}
Let $0<\mu\ll(c/10)^{10/c}$ and let $N=4n+\mu n$.

Let $T$ be a tree with bipartition class sizes $2n$ and $2n-\mu n$, such that it contains two adjacent vertices $x$ and $y$ with $d(x)+d(y)-2=10/c$, and all other vertices of $T$ are distributed as evenly as possible as neighbours of the $10/c$ vertices in $W=(N(x)\cup N(y))\setminus\{x,y\}$. In particular, $\Delta(T)\leq cn$. Let $\cT$ be a tree with bipartition class sizes $2n$ and $\mu n$, such that it contains a vertex $z$ with degree $\mu n$, and all other vertices of $\cT$ are neighbours of the $\mu n$ neighbours of $z$, distributed so that $\Delta(\cT)=cn$. Then,~\ref{construct:3:1}--\ref{construct:3:3} hold. 

Let $q=c/10$, $p=1-q$, and let $G\sim G(N,p)$. Then, by Lemma~\ref{lemma:chernoff}, with high probability $G$ has minimum degree at least $(4+\mu)n-cn/2$.

Let $U$ be any fixed subset of $V(G)$ of size $t=10/c$, and let $X_U$ be the number of vertices in $V(G)\setminus U$ adjacent to at least one vertex in $U$. Then, the expected value of $X_U$ is
\[\mathbb{E}[X_U]\leq(4+\mu)n(1-q^t)\leq(4-2\mu)n,\]
using that $\mu\ll(c/10)^{10/c}=q^t$. By Lemma~\ref{lemma:chernoff}, $\mathbb{P}(X_U\geq(4-\mu)n-t)\leq2\exp(-\mu^2n/1000)$. Therefore, by a union bound, the probability that there exists $U\subset V(G)$ of size $t$ with $|N(U)\cup U|=X_U+t\geq(4-\mu)n=|T|$ is at most 
\[2\binom{N}t\exp(-\mu^2n/1000)\leq 2(5n)^{10/c}\exp(-\mu^2n/1000)=o(1).\]
Hence, with high probability such $U$ does not exist. 

Fix a realisation of $G$ such that $\delta(G)\geq(4+\mu)n-cn/2$, and there is no $U\subset V(G)$ of size $10/c$ satisfying $|N(U)\cup U|\geq|T|$. Then, $G$ contains no copy of $T$ as $T$ contains a set $W$ of $10/c$ vertices satisfying $|N(W)\cup W|=|T|$, while $G^c$ contains no copy of $\cT$ as $\Delta(\cT)=cn>|G|-\delta(G)\geq\Delta(G^c)$. Therefore, $R(T,\cT)\geq 4n+\mu n$, and so~\ref{construct:3:4} holds.
\end{proof} 

Finally, note that the assumptions $n\geq\nu$ and $\nu\geq t_1$ in Theorem~\ref{thm:main} imply that $n\geq\nu\geq n/2$, so for all positive results in this paper, the sizes of $T$ and $\cT$ are close. What happens when $|T|$ is significantly larger than $|\cT|$? Are there families of pairs $(T,\cT)$ for which the lower bound (\ref{eq:generallower}) in Proposition~\ref{prop:generallower} is tight? One positive result in this direction is a special case of {\cite[Theorem 1.1]{MPY2}}, which establishes the Ramsey goodness of bounded degree trees versus general graphs. 
\begin{theorem}\label{thm:goodness}
For every $\Delta\geq1$, there exists a constant $C_\Delta\gg1$ such that the following holds. 

For every $\nu$-vertex tree $\cT$ with bipartition class sizes $\tau_1\geq\tau_2$, every $n\geq C_\Delta\nu$, and every $n$-vertex tree $T$ with $\Delta(T)\leq\Delta$, 
\[R(T,\cT)=\underline{R}(T,\cT)=n+\tau_2-1.\]
\end{theorem}

\begin{question}
Decide for more families of pairs of trees $(T,\cT)$ whether $R(T,\cT)=\underline{R}(T,\cT)$.
\end{question}

\bibliography{bibliography}

@book{JLR,
	    AUTHOR = {Janson, Svante and {\L}uczak, Tomasz and Ruci{\'n}ski, Andrzej},
	TITLE = {Random graphs},
	PUBLISHER = {Wiley-Interscience},
	YEAR = {2000},
	PAGES = {xii+333},
	ISBN = {0-471-17541-2},
	MRCLASS = {05C80 (60C05 82B41)},
	MRNUMBER = {1782847},
	MRREVIEWER = {Mark R. Jerrum},
	DOI = {10.1002/9781118032718},
	URL = {https://doi.org/10.1002/9781118032718},
}

@article{HLT,
  title={{R}amsey numbers for trees of small maximum degree},
  author={Haxell, Penny E. and {\L}uczak, Tomasz and Tingley, Peter W.},
  journal={Combinatorica},
  volume={22},
  number={2},
  pages={287--320},
  year={2002},
  publisher={Springer-Verlag GmbH}
}

@article{MPY,
  title={Ramsey numbers of trees},
  author={Montgomery, Richard and Pavez-Sign{\'e}, Mat{\'\i}as and Yan, Jun},
  journal={arXiv:2509.07934},
  year={2025}
}

@article{Y,
  title={Spanning trees with large maximum degrees},
  author={Yan, Jun},
  journal={arXiv:2510.17736},
  year={2025}
}

@article{H,
  title={On Representatives of Subsets},
  author={Hall, Philip},
  journal={Journal of the London Mathematical Society},
  volume={s1-10},
  number={1},
  pages={26--30},
  year={1935},
  publisher={Wiley Online Library}
}

@incollection{W,
  title={The differential equation method for random graph processes and greedy algorithms},
  author={Wormald, Nicholas C.},
  booktitle={Lectures on Approximation and Randomized Algorithms},
  pages={73--155},
  year={1999},
  publisher={Polish Scientific Publishers}
}

@book{Bo,
  title={Modern Graph Theory},
  author={Bollob{\'a}s, B{\'e}la},
  isbn={9780387984889},
  lccn={98011960},
  url={https://books.google.co.uk/books?id=SbZKSZ-1qrwC},
  year={1998},
  publisher={Springer}
}

@article{K,
author = {Krivelevich, Michael},
title = {Embedding Spanning Trees in Random Graphs},
journal = {SIAM Journal on Discrete Mathematics},
volume = {24},
number = {4},
pages = {1495-1500},
year = {2010},
doi = {10.1137/100805753},}

@article{Ra,
  title={On a Problem of Formal Logic},
  author={Frank P. Ramsey},
  journal={Proceedings of The London Mathematical Society},
  volume={s2-30},
number={1},
  pages={264--286},
year={1930}
  }

@article{GG,
  title={On {R}amsey-type problems},
author = {L\'aszl\'o Gerencs\'er and Andr\'as Gy\'arf\'as},
journal ={Annales Universitatis Scientiarium Budapestinensis de Rolando Eötvös Nominatae,
Sectio Mathematica},
volume ={10},
pages ={167--170},
year={1967}
}

@article{BR,
  title={On {R}amsey numbers for stars},
  author={Burr, Stefan A. and Roberts, John A.},
  journal={Utilitas Mathematica},
  volume={4},
  pages={217--220},
  year={1973}
}

@article{FS,
  title={All {R}amsey numbers for cycles in graphs},
  author={Faudree, Ralph J. and Schelp, Richard H.},
  journal={Discrete Mathematics},
  volume={8},
  number={4},
  pages={313--329},
  year={1974},
  publisher={Elsevier}
}

@article{R,
  title={On a {R}amsey-type problem of {J. A.} {B}ondy and {P}. {E}rd{\"o}s. {II}},
  author={Rosta, Vera},
  journal={Journal of Combinatorial Theory, Series B},
  volume={15},
  number={1},
  pages={105--120},
  year={1973},
  publisher={Elsevier}
}

@Incollection{B,
author="Burr, Stefan A.",
title="Generalized {R}amsey theory for graphs - a survey",
booktitle="Graphs and Combinatorics",
year="1974",
publisher="Springer",
pages="52--75",
isbn="978-3-540-37809-9"
}

@article{GHK,
title = {Generalized {R}amsey theory for graphs, x: double stars},
journal = {Discrete Mathematics},
volume = {28},
number = {3},
pages = {247-254},
year = {1979},
issn = {0012-365X},
doi = {https://doi.org/10.1016/0012-365X(79)90132-8},
url = {https://www.sciencedirect.com/science/article/pii/0012365X79901328},
author = {Jerrold W. Grossman and Frank Harary and Maria Klawe}}

@article{NSZ,
  title={Asymptotics of {R}amsey numbers of double stars},
  author={Norin, Sergey and Sun, Yue Ru and Zhao, Yi},
  journal={arXiv:1605.03612},
  year={2016}
}

@article{MPY2,
  title={Ramsey numbers of bounded degree trees versus general graphs},
  author={Montgomery, Richard and Pavez-Sign{\'e}, Mat{\'\i}as and Yan, Jun},
  journal={Journal of Combinatorial Theory, Series B},
  volume={173},
  pages={102--145},
  year={2025},
  publisher={Elsevier}
}

@incollection {komlos1995szemeredi,
    AUTHOR = {Koml{\'o}s, J{\'a}nos and Simonovits, Mikl{\'o}s},
     TITLE = {Szemer{\'e}di's {R}egularity {L}emma and its applications in graph
              theory},
 BOOKTITLE = {Combinatorics, {P}aul {E}rd{\H o}s is eighty, {V}olume 2},
     PAGES = {295--352},
 PUBLISHER = {J{\'a}nos Bolyai Mathematical Society},
      YEAR = {1996},
}
\bibliographystyle{plain}
\end{document}